\documentclass[a4paper,10pt]{article}

\usepackage{graphicx}
\usepackage{amssymb}
\usepackage{amsthm}
\usepackage{amsmath}
\usepackage{enumitem}

\usepackage[dvips]{hyperref}
\usepackage{breakurl}

\usepackage{enumitem}
\usepackage{bookman}

\newtheorem{theorem}{Theorem}
\newtheorem{definition}[theorem]{Definition}
\newtheorem{remark}[theorem]{Remark}
\newtheorem{example}[theorem]{Example}
\newtheorem{corollary}[theorem]{Corollary}
\newtheorem{proposition}[theorem]{Proposition}
\newtheorem{lemma}[theorem]{Lemma}

\newtheorem*{proof1}{Proof of Theorem~\ref{thm:main}}

\numberwithin{theorem}{section}
\numberwithin{equation}{section}
\numberwithin{figure}{section}

\newcommand{\data}{V}
\newcommand{\signal}{u}
\newcommand{\noise}{\eps}

\newcommand{\Data}{Y}
\newcommand{\Para}{x}
\newcommand{\Noise}{\xi}

\newcommand{\base}{\phi}
\newcommand{\dbase}{\tilde\phi}
\newcommand{\curve}{\psi}
\newcommand{\wave}{\psi}
\newcommand{\scale}{\phi}
\newcommand{\Base}{\boldsymbol \Phi}
\newcommand{\dict}{\mathcal D}
\newcommand{\wi}{w}
\newcommand{\J}{\mathcal J}

\newcommand{\ball}{R}
\newcommand{\edot}{\,\cdot\,}

\newcommand{\kl}[1]{\left(#1\right)}

\newcommand\bkl[1]{\Bigl(#1\Bigr)}
\newcommand{\ekl}[1]{\left[#1\right]}
\newcommand\set[1]{\left\{#1\right\}}

\newcommand\norm[1]{\left\Vert#1\right\Vert}
\newcommand\enorm{\lVert\,\cdot\,\rVert}
\newcommand\einner{\langle\,\cdot\,,\,\cdot\,\rangle}
\newcommand\abs[1]{\left\lvert#1\right\rvert}
\newcommand\xsabs[1]{\lvert#1\rvert}
\newcommand\inner[2]{\left\langle  #1,#2 \right\rangle}
\newcommand{\ckl}[1]{\left\lceil#1\right\rceil}
\newcommand{\fkl}[1]{\left\lfloor#1\right\rfloor}
\newcommand{\req}[1]{(\ref{eq:#1})}

\newcommand\skl[1]{(#1)}
\newcommand\mkl[1]{\big(#1\bigr)}
\newcommand\sset[1]{\{#1\}}
\newcommand\mset[1]{\bigl\{#1\bigr\}}
\newcommand\snorm[1]{\lVert#1\rVert}
\newcommand\mnorm[1]{\bigl\lVert#1\bigr\rVert}
\newcommand\sinner[2]{\langle  #1,#2 \rangle}
\newcommand\minner[2]{\bigl\langle  #1,#2 \bigr\rangle}
\newcommand\sabs[1]{\lvert#1\rvert}
\newcommand\mabs[1]{\bigl\lvert#1\bigr\rvert}

\newcommand\R{\mathbb{R}}
\newcommand\Z{\mathbb{Z}}
\newcommand\N{\mathbb{N}}

\newcommand\I{I}

\newcommand\B{\mathcal{B}}

\newcommand{\om}{\omega}
\newcommand{\la}{\lambda}
\newcommand{\Om}{\Omega}
\newcommand{\La}{\Lambda}
\newcommand{\eps}{\epsilon}

\DeclareMathOperator{\wk}{\mathbf P}
\DeclareMathOperator{\ew}{\mathbf E}
\DeclareMathOperator{\cov}{\mathbf{Cov}}

\DeclareMathOperator{\range}{\mathrm Ran}

\newcommand{\soft}{S}
\newcommand{\softop}{\mathbf S}
\newcommand{\shiftop}{\mathbf T}
\newcommand{\Ft}{\mathbf F}

\newcommand{\Ct}{\mathbf C}
\newcommand{\Wave}{\mathbf W}
\newcommand{\Scale}{\mathbf V}

\newcommand{\V}{\mathcal V}
\newcommand{\W}{\mathcal W}

\usepackage{geometry}
\title{Extreme Value Analysis of Empirical Frame Coefficients
and Implications for Denoising by\\Soft-Thresholding}

\author{Markus Haltmeier{}$^{\star,}\footnote{Corresponding author.
E-mail address:
\href{mailto:markus.haltmeier@uibk.ac.at}
{\tt markus.haltmeier@uibk.ac.at} }$
\quad Axel Munk{}$^{\diamond,\sharp}$
}

\date{\normalsize ${}^\star$Department of Mathematics\\University of Innsbruck\\
Technikerstra{\ss}e 21a, A-6020  Innsbruck
\\[0.5em]
${}^\diamond$Statistical Inverse Problems in Biophysics\\Max Planck Institute for Biophysical Chemistry\\
Am Fa{\ss}berg 11, D-37077 G\"ottingen
\\[0.5em]
${}^\sharp$Institute for Mathematical Stochastics\\University of G\"ottingen,\\
Goldschmidtstra{\ss}e 7, D-37077 G\"ottingen}
\begin{document}
\maketitle

\begin{abstract}
Denoising by frame thresholding  is one of the most basic and
efficient methods for recovering a discrete signal or image
from data  that are corrupted by additive Gaussian white noise.
The basic  idea is to select a frame of analyzing elements
that separates the  data in few large coefficients due to the signal
and many small coefficients mainly due to
the noise $\noise_n$.
Removing  all  data coefficients  being in magnitude below a certain threshold yields a reconstruction
of the original signal.
In order to properly  balance the amount of noise to be removed and the relevant signal features to be kept,
a precise understanding of  the statistical properties
of  thresholding  is important. For that purpose we derive
the asymptotic distribution of  $\max_{\om\in\Om_n} \sabs{\sinner{\base_{\om}^{n}}{\noise_n}}$  for a wide class
of  redundant frames $\mkl{\base_{\om}^{n}: \om \in \Om_n}$.
Based on  our theoretical  results  we give a rationale for
universal extreme value thresholding techniques
yielding asymptotically sharp confidence regions and smoothness estimates corresponding  to prescribed significance levels.
The results cover many frames used in
imaging and signal recovery applications, such as redundant wavelet
systems, curvelet frames, or unions of bases.
We show that `generically' a standard Gumbel
law results as it is  known from  the case of  orthonormal wavelet bases.
However,  for specific highly redundant frames other limiting
laws may occur. We indeed verify that the translation invariant wavelet transform shows a different asymptotic behaviour.

\end{abstract}

\section{Introduction}
\label{sec:intro}

We consider the problem of estimating  a  $d$-signal or image
$\signal_n$ from noisy observations
\begin{equation} \label{eq:problem}
	\data_n \kl{k} = \signal_n \kl{k} +  \noise_n\kl{k}
	\,,
	\qquad \text{ for } \;
	k \in \I_n := \set{0,\dots,  n - 1}^{d}  \quad  \text{ with }  d \in \N  \,.
\end{equation}
Here   $\noise_n\kl{k} \sim N\mkl{0,\sigma^2}$ are independent normally distributed random variables (the noise), $n$ is the level of discretization,
and $\sigma^2 $ is the  variance of the data  (the  noise level).
The  signal $\signal_n$ is assumed to be  a discrete approximation  of some underlying continuous domain  signal obtained by  discretizing a function $\signal \colon  [0,1]^d \to \R$.
One may think of  the entries  of $\signal_n$ as point samples
$\signal_n \kl{k} = \signal\kl{k/n}$ on an equidistant
grid. However, in some situations it may be more realistic to consider other discretization models. Area samples, for example,  are  more  appropriate in many imaging applications.
In this paper we will not pursue this topic further, because most of the presented results do not
crucially depend on the particular discretization model as long as $\signal_n$ can be associated with a function
$\signal_n^* \colon  [0,1]^d \to \R$
(some kind of abstract interpolation) which, in a suitable way,
tends to  $\signal$ as $n \to \infty$.

The aim of denoising  is to estimate the unknown signal
$\signal_n  :=  \mkl{\signal_n\kl{k}: k\in \I_n} \in \R^{\I_n}$  from  the data $\data_n :=  \mkl{\data_n\kl{k}: k\in \I_n}
\in \R^{\I_n}$.
The particular estimation procedure we will analyze in detail is soft-thresholding in frames and overcomplete dictionaries.
We stress, however, that a similar analysis also applies to
different thresholding methods, such as block thresholding
techniques (as considered, for example, in \cite{Cai99,CaiZho09,CheFadSta10,HalEtal97}).

\subsection{Wavelet Soft-Thresholding}
\label{sec:wave-soft}

In order to motivate our results for thresholding for general
frames we start by one dimensional wavelet  soft-thresholding.
For that purpose, let  $\mkl{\wave_{j,k}^{n} \colon
(j,k) \in \Om_n}$ denote an orthonormal wavelet basis of $\R^n$,
where $n = 2^J$ is the number of data points  and
\begin{equation*}
    \Om_n := \set{\skl{j,k}:
    j \in \sset{ 0, \dots,  J  -1} \text{ and }
    k \in \sset{0, \dots, 2^j-1 }}
\end{equation*}
the index set of the wavelet basis.
Wavelet soft-thresholding is by now a standard method for
signal and image denoising
(see, for example, \cite{AntBigSap01,CoiDon95,Don95,DonJoh94,DonJohKerPic96,HaeKerPicTsy98,Joh11,Mal09,StaMurFad10} for surveys and
some original references).
It consists of the following  three basic steps:

\begin{enumerate}[label=\arabic*), topsep=0em]
\item %\emph{Wavelet analysis.}
\label{it:thresh1}
Compute all empirical wavelet coefficients
$\Data_n \kl{j,k} = \minner{\wave_{j,k}^{n}}{\data_n}$ of the given noisy
data with respect to the considered
orthonormal  wavelet basis.

\item \label{it:thresh2}
For some threshold $T_n \in \kl{0,\infty}$, depending on the noise level and the number of data points,
apply the nonlinear soft-thresholding function
\begin{equation}\label{eq:soft}
\soft \mkl{\edot, T_n} \colon \R \to \R \colon
y \mapsto \soft\kl{y, T_n}
:=
\begin{cases}
y  + T_n & \text{ if } y  \leq   -T_n \\
y-  T_n & \text{ if } y  \geq   \phantom{-}T_n \\
0    & \text{ otherwise }
\end{cases}
\end{equation}
to each wavelet  coefficient of the data.
The resulting thresholded coefficients
$\soft  \skl{\Data_n(j,k), T}$ are then considered
as estimates for the wavelet coefficients
of $\signal_n$.
Notice, that the soft-thresholding function can be written in the
compact form $\soft \kl{y, T_n} = \operatorname{sign}\kl{y} \kl{\abs{y} - T_n}_+$. Further,
it sets all coefficients being in magnitude smaller than $T_n$ to zero and shrinks the
remaining coefficients towards zero by the value $T_n$.

\item \label{it:thresh3}
The desired estimate for the signal $\signal_n$ is then
defined  by  the wavelet series of the   thresholded empirical coefficients
$\soft  \skl{\Data_n(j,k), T}$,
\begin{equation} \label{eq:wave-soft}
\hat \signal_n
= \sum_{j=0}^{J}\sum_{k=0}^{2^j-1}
\soft  \skl{\Data_n(j,k), T_n}
\,
\wave^n_{j,k} \,.
\end{equation}
\end{enumerate}

Every step  in the above procedure can  be computed in
$\mathcal O (n)$ operation counts and hence
the overall procedure  of wavelet soft-thresholding is linear in the  number
$n$ of unknown parameters (see \cite{Coh03,DonJoh94,Mal09}).
It is  thus not only conceptually simple but also allows for
fast numerical implementation. Even simple linear  spectral denoising techniques  using  the FFT algorithm have a numerical complexity of
$\mathcal O\mkl{ n \log n}$ floating point
operations. Besides  these practical advantages, wavelet
soft-thresholding also obeys certain theoretical optimality properties.
It  yields  to an almost optimal mean square error simultaneously  over a wide range of  function spaces (including Sobolev and Besov spaces) and, at the same time, has a  smoothing effect  with respect to any of the norms in
these spaces. Hence soft-thresholding  automatically
adapts to the unknown smoothness of the desired signal~\cite{Don95,DonJoh95}.

%\subsection{Denoising Property of  the Universal Threshold}
%\label{sec:choice}

Any  particular choice of the  thresholding parameter
$T_n$ is a tradeoff between signal
approximation and noise reduction:
A large threshold  removes much of  the noise but also removes parts of
the signal. Hence  a reasonable threshold choice  should be as small as
possible under the side constrained that a significant amount of the
noise is removed. The smaller the actual threshold is taken, the more emphasis is given on signal representation and the less emphasis on noise reduction.
A commonly used threshold is the so called \emph{universal threshold} $T_n =  \sigma \sqrt{2 \log n}$ as proposed in the seminal work~\cite{Don95},
where the following result is  shown.

\begin{theorem}[Denoising property of wavelet soft-thresholding~\cite{Don95}]\label{thm:donoho}\mbox{}\\
Suppose that $\dict_n$ are consistent with an underlying orthonormal wavelet basis
$\dict$ on $[0,1]$ having $m$-times continuously differentiable elements and
$m$ vanishing moments,
that $\signal_n \kl{k}= u \kl{k/n}$, for $k=0, \dots n-1$, denote point samples  of a function $u \colon [0,1] \to \R$ and that $\hat \signal_n$ are constructed by~\req{wave-soft} with the universal threshold $T_n = \sigma \sqrt{2 \log n}$.
Then, there exists a special smooth interpolation of $\hat \signal_n$
producing a function $u_n^* \colon [0,1] \to \R$.
Further, there are universal constants $\skl{\pi_n}_n \subset (0,1)$
with $\pi_n \to 1$ as $n = 2^J \to \infty$, such that for any Besov space
$\mathcal B_{p,q}^r$ which embeds continuously
into $C[0,1]$ (hence $r>1/p$) and for which $\dict$ is an unconditional basis (hence $r<m$),
\begin{equation}\label{eq:wavelet-denoise}
 	\wk \set{\norm{\signal_n^*}_{\mathcal B_{p,q}^r}
	\leq c(\mathcal B_{p,q}^r, \dict) \norm{\signal}_{\mathcal B_{p,q}^r} \,; \forall u \in \mathcal B_{p,q}^r }
	\geq  \pi_n   \,,
\end{equation}
for constants $c(\mathcal B_{p,q}^r, \dict)$
depending on  $\mathcal B_{p,q}^r$ and $\dict$ but neither
on $u$ nor on  $n$.
\end{theorem}

Theorem~\ref{thm:donoho} states that the estimate $\signal_n^*$ is, with
probability tending to one, simultaneously  as smooth as $\signal$
for all smoothness spaces $\mathcal B_{p,q}^r$.
This result can be derived from the denoising  property
(see \cite{Don95,Joh11,JohSil97} and also Section~\ref{sec:denoising})
\begin{equation}  \label{eq:noise-remove}
	\wk \set{
	\max_{(j,k) \in \Om_n} \abs{ \minner{\wave_{j,k}^{n}}{\noise_n}}
	\leq    \sigma \sqrt{2 \log n} }
	\geq \pi_n \to 1 \qquad \text{ as }
	 n \to \infty \,.
\end{equation}
For an orthonormal basis, the noise coefficients
$\sinner{\wave_{j,k}^{n}}{\noise_n} \sim N (0, \sigma^2)$ are
independently distributed. Hence Equation~\req{noise-remove} is a consequence from standard extreme value results for  independent normally distributed
random vectors \cite{DehFer06,LeaLinRoo83}.
Extreme value theory also implies the limiting Gumbel law
\begin{equation} \label{eq:gumbel-wave}
\lim_{n \to \infty}
\wk\set{ \max_{\om \in \Om_n} \abs{\minner{\base_{\om}^{n}}{\noise_n}}
\leq
\sigma  \sqrt{2 \log n}
+
\sigma \,\frac{2z-  \log\log n - \log\pi}{2\sqrt{2 \log n}}}
=
\exp\kl{-e^{-z} } \,,
\end{equation}
uniformly in  $z \in \R$. This even allows to exactly characterize all thresholding sequences
$T_n$  yielding a denoising property like \req{noise-remove} with $T_n$
in place of $\sigma \sqrt{2 \log n}$.

In the case that a redundant frame is considered instead of an orthonormal wavelet
basis,  then the empirical coefficients are no more linear independent and
limiting result like \req{gumbel-wave} are much harder to come up with.
In this paper we verify that a similar distributional result as in  \req{gumbel-wave}
holds for a wide class of redundant frames
with $n$ replaced by the number of frame elements.
This class  is shown to include non-orthogonal wavelets, curvelet frames and
unions of bases (see Theorems~\ref{thm:wave}, \ref{thm:cs} and \ref{thm:curve}).
Roughly speaking, the reason is, that the  redundancy is of these frames weak enough that it asymptotically
vanishes in a statistical sense and the system behaves as an independent system.
 However, we also an important  example (in the form of the  translational wavelet system; see Theorem~\ref{thm:ti}) which shows that highly redundant systems may show a different asymptotic behaviour.

Our work is motivated by the well known observation that the universal threshold sigma $ \sigma \sqrt{2 \log n }$
often is found to be  too  large in applications, hence including too few coefficients into the final estimator (see \cite{AntFan01,DonJoh95,Mal09}).  This recently has initiated further research on refined thresholding methods and we would like to shed some light on this phenomenon for a large class of frame systems by providing a refined  asymptotics as in \req{gumbel-wave} in addition to results of the type \req{noise-remove}.
We also provide a selective review on current thresholding methodology where we focus on the link between
statistical extreme value theory and thresholding techniques.

\subsection{Frame Soft-Thresholding: Main Results}
\label{sec:main}

For any  $n \in \N$, let $\dict_n = \mkl{\base_{\om}^{n}: \om\in\Om_n}$ denote  a frame of $\R^{\I_n}$, where $\Om_n$ is a finite index set, that consists of
normalized frame elements (that is, $\snorm{\base_{\om}^{n}}=1$ holds for all $\om \in \Om_n$) and has frame bounds $a_n \leq b_n$ (compare Section~\ref{sec:notation}). Our main results concerning  thresholding
estimation in the frame $\dict_n$ will hold for  asymptotically stable frames,
which  are defined as follows.

\begin{definition}[Asymptotically stable frames]\label{def:stab}\mbox{}\\
We say that a family of frames $\skl{\dict_n}_{n \in \N}$ with normalized frame elements is
\emph{asymptotically stable}, if the following  assertions
hold true:
\begin{enumerate}[label=(\roman*),topsep=0em]
\item \label{it:frame2}
For some  $\rho \in\kl{0,1}$,
$\abs{\sset{\skl{\om,\om'} \in \Om_n^2:
\sabs{\sinner{\base_{\om}^{n}}{\base_{\om'}^{n}}} \geq  \rho }  }
= o\kl{\frac{\abs{\Om_n}}{\sqrt{\log \abs{\Om_n}}}}$ as $n\to \infty$.

\item \label{it:frame3}
The upper frame bounds $b_n$ are uniformly bounded, i.e., $B:= \sup\set{ b_n: n \in \N}< \infty$.
\end{enumerate}
\end{definition}

The following Theorem~\ref{thm:main} is the key to most results  of this paper.  It states,   that after proper normalization the distribution of $\max_{\om \in \Om_n} \sabs{\sinner{\base_{\om}^{n}}{\noise_n}}$ converges to the Gumbel distribution -- provided that   the frames are asymptotically stable.

\begin{theorem}[Limiting distribution for asymptotically stable frames]
\label{thm:main}\mbox{}\\
Assume that $\skl{\dict_n}_{n \in \N}$ is
an asymptotically stable family of frames, and let $\skl{\noise_n}_{n\in\N}$ be a sequence of  random vectors in $\R^{\I_n}$ with independent $N\skl{0,\sigma^2}$-distributed entries.
Then, for every $z \in \R$,
\begin{equation}
\label{eq:gumbel-frame}
\lim_{n \to \infty}
\wk\set{ \max_{\om \in \Om_n} \abs{\minner{\base_{\om}^{n}}{\noise_n}}
\leq
\sigma  \sqrt{2 \log \abs{\Om_n}}
+
\sigma \,\frac{2z-  \log\log \abs{\Om_n} - \log\pi}{2\sqrt{2 \log \abs{\Om_n}}}}
=
\exp\kl{-e^{-z} } \,.
\end{equation}
The function $z\mapsto \exp\mkl{-e^{-z} }$
is known as the  Gumbel distribution.
\end{theorem}

\begin{proof}
See Section~\ref{sec:thm:main}.
\end{proof}

In the case that $\dict_n$ are orthonormal bases,
results similar to the one of Theorem~\ref{thm:main} follow
from standard extreme value results
(see, for example, \cite{DehFer06,LeaLinRoo83}) and are  well
known in the wavelet  community (see, for example,~\cite{Don95,Joh11,Mal09}).
However,  neither the convergence  of  the maxima  including
absolute values (which is the actually relevant case)
nor the use of redundant systems are  covered by these results.
In Section~\ref{sec:ez} we shall verify that  many
redundant frames, such as non-orthogonal wavelet frames, curvelets and unions
of bases,  are asymptotically stable and hence the limiting Gumbel law of Theorem~\ref{thm:main} can be applied.
Based on this limiting extreme value distribution we provide
an exact characterisation of all thresholds $T_n$ satisfying
a denoising property similar to the one of \req{noise-remove}
for general frame thresholding;  see Section~\ref{sec:evd-frame}
for details.

Suppose that  $(\alpha_n)_{n \in \N}$ is a  sequence in $(0,1)$  converging
to $\alpha  \in [0,1)$, let $z_n$ satisfy $\exp\kl{-e^{-z_n}} = \alpha_n$  and let $ \signal_n$ denote the
 wavelet soft-thresholding estimate with the  threshold
 \begin{equation}\label{eq:evt-wave}
T_n
 =
\sigma \sqrt{2 \log \abs{\Om_n}} +
\sigma  \, \frac{2 z_n  -   \log\log
\abs{\Om_n} - \log \pi}{ 2\sqrt{2 \log \abs{\Om_n}}}\,.
\end{equation}
According to Theorem~\ref{thm:main}, the probability that $\signal_n$ is contained in
$
R_n=  \sset{ \bar \signal_n
\colon \sabs{\sinner{\base_{j,k}}{\data_n - \bar \signal_n}  }  \leq T_n \text{ for all } \om \in \Om_n}
$ tends to $1-\alpha$ as $n \to \infty$.
Hence the sets $R_n$ define asymptotically sharp confidence regions around the given data for any significance level $\alpha$;
see Section~\ref{sec:alphathresh} for details.

The proof  of Theorem~\ref{thm:main}
relies on new extreme value results for dependent
chi-square distributed random variables (with one degree of freedom) which we establish  in Section~\ref{ap:evd}. In the field of statistical extreme value
theory, the following definition is  common.

\begin{definition}[Gumbel type]\label{def:gumbeltype}\mbox{}\\
A sequence  $\kl{M_n}_{n \in \N}$
of real valued random variables  is said to be of  \emph{Gumbel type}
(or to be of extreme value typ I),  if there are real valued normalizing sequences
$\kl{a_n}_{n \in \N}$ and $\kl{b_n}_{n \in \N}$, such that the
limit $\wk \set{ M_n  \leq a_n  z + b_n } \to \exp\mkl{-e^{-z} }$  as $n\to \infty$ holds point-wise for all $z \in \R$ (and therefore uniformly).
\end{definition}

Using the notion just introduced, Theorem~\ref{thm:main} states that
$\max_{\om \in \Om_n} \sabs{\sinner{\base_{\om}^{n}}{\noise_n} }$ is of
Gumbel type,  with normalizing sequences $\sigma a\mkl{\chi,\abs{\Om_n}}$ and $\sigma b\mkl{\chi,\abs{\Om_n}}$, where
\begin{align}\label{eq:a-abs}
a\mkl{\chi,\abs{\Om_n}}  &:=  \frac{1}{\sqrt{2 \log \abs{\Om_n}}} \,,
\\ \label{eq:b-abs}
b\mkl{\chi,\abs{\Om_n}}  &:= \sqrt{2 \log \abs{\Om_n}} -
 \frac{ \log \log \abs{\Om_n} +   \log \pi} {2 \sqrt{2\log \abs{\Om_n}}} \,.
\end{align}
As shown in Theorem~\ref{thm:mainb}, the  maxima
of $\sinner{\base_{\om}^{n}}{\noise_n}$
without taking absolute  values are also of Gumbel type.
We emphasize, however, that the   corresponding  normalizing sequences differ from those
required for the maxima
with absolute values.
Indeed,   $\max_{\om \in \Om_n} \sabs{\sinner{\base_{\om}^{n}}{\noise_n}}$ behaves as the maximum of  $2 \abs{\Om_n}$ (opposed to $\abs{\Om_n}$) independent standard normally  distributed random  variables; compare with Remark~ \ref{rem:normalizing}.
The different fluctuation behaviour of  the maxima with and without absolute values
is not resembled by Equation~\req{noise-remove}, which is exactly
the same for the maxima with and without absolute values.
Only in a  refined distributional  limit \req{gumbel-frame} this difference
becomes visible.
Moreover, in the case that the frames $\dict_n$ are redundant, no result similar
to Theorem~\ref{thm:main} is known at all.

Asymptotical stability typically fails for frames  without an underlying infinite dimensional frame.
A prototype for  such a family is the  dyadic  translation invariant wavelet transform (see Section~\ref{sec:ti}).
In this case, the redundancy of  the translation invariant wavelet system
increases boundlessly with increasing $n$,  which implies that the corresponding upper frame bounds tend to infinity as $n \to \infty$.
We indeed prove the following counterexample
if  Condition~\ref{it:frame3} in Definition~\ref{def:stab}
fails to hold.

\begin{theorem}[Tighter bound for translation invariant wavelet systems] \label{thm:ti2}\mbox{}\\
Suppose that $\skl{\wave_\om}_{\om \in \Om_n}$ is a
discrete translation invariant wavelet system with unit norm elements generated by a mother wavelet $\wave$ that is continuously
differentiable, and let $\skl{\noise_n}_{n\in\N}$ be a sequence of  random vectors in $\R^{\I_n}$ with independent $N\skl{0,\sigma^2}$-distributed entries.
Then, for some constant $c>0$ and all $z\in \R$
we have
\begin{equation*}
\liminf_{n \to \infty}
\wk\set{ \max_{\om \in \Om_n} \abs{\minner{\wave_{\om}^{n}}{\noise_n}}
\leq
\sqrt{2 \log n}
+
\frac{z + \log \kl{ c/\pi}}{\sqrt{2\log n}}
}
\geq  \exp\kl{-e^{-z} }
 \,.
\end{equation*}
\end{theorem}

\begin{proof}
This follows from Theorem~\ref{thm:ti},
that we  proof  in Section~\ref{ap:thm:ti}.
\end{proof}

Theorem \ref{thm:ti2} shows that the
maximum of the translation invariant wavelet
coefficients is strictly smaller (in a distributional sense; see Section~\ref{sec:ti})
than the maximum of
an asymptotically stable frame with $\abs{\Om_n} = n \log n$
elements and therefore the result of Theorem~\ref{thm:main} does not hold for a
translation invariant wavelet system. Moreover, Theorem~\ref{thm:ti2}
shows that there exists a thresholding sequence being  strictly smaller than
$\sqrt{2 \log \abs{\Om_n}}$  yields asymptotic
smoothness; see Section~\ref{sec:ti} for details.
This also reveals the necessity of a detailed extreme value
analysis of the empirical noise coefficients in the case of
redundant frames.

\subsection{Outline}

In the following Section~\ref{sec:thresh} we introduce some notation used throughout this paper. In particular, we define the soft-thresholding estimator in redundant frames.
The core part of this paper is Section~\ref{sec:evd-frame}, where we proof the asymptotic distribution  of the frame coefficients claimed in  Theorem~\ref{thm:main}. This result is then applied to define  extreme value based thresholding rules and corresponding sharp confidence regions.
Moreover, in this section we  show that the resulting thresholding estimators satisfy both, oracle inequalities for the mean square error and smoothness estimates for a wide class of smoothness measures.
Our proofs require new facts from statistical extreme value theory for the maxima of  absolute values of dependent normal random variables that we derive in
Section~\ref{ap:evd}.
Finally, in section in Section~\ref{sec:ez} discuss in detail several examples, including, non orthogonal frames,  biorthogonal wavelets, curvelets and unions of bases in detail

\section{Thresholding in Redundant Frames}
 \label{sec:thresh}

For the following recall the model~\req{problem}
and write $\noise_n = \mkl{\noise_n\kl{k}: k\in \I_n} \in \R^{\I_n}$ for the noise vector in~\req{problem}.
We assume throughout that the variance $\sigma^2$ of the noise
is given. Fast and efficient methods for estimating the variance are well known
(see, for example,  \cite{BroLev07,DetMunWag98,HalKayTit90,HalMur90} for $d=1$
 and \cite{MunBisWagFre05}  for $d  \geq 2$).

Throughout this paper all estimates  for the signal $\signal_n$
are based on thresholding  the coefficients of the given data $\data_n$
with respect to prescribed frames of analyzing elements.

\subsection{Frames}
\label{sec:notation}

In the sequel $\dict_n := \mkl{\base_{\om}^{n}: \om\in\Om_n}
\subset \R^{\I_n}$
denotes a frame of  $\R^{\I_n}$, with $\Om_n$ being a finite index set. Hence  there exist constants
$0< a_n \leq b_n< \infty$, such that
\begin{equation} \label{eq:frame-prop}
	\kl{\forall \signal_n \in \R^{\I_n}}
	\qquad
	a_n \mnorm{\signal_n}^2 \leq
	\sum_{\om \in \Om_n}
	\mabs{\minner{\base_{\om}^{n}}{\signal_n}}^2
	\leq b_n \mnorm{\signal_n}^2 \,.
\end{equation}
(Here  $\enorm$ is the Euclidean  norm on $\R^{\I_n}$ and
$\einner$ the corresponding inner product.) The largest  and
smallest numbers $a_n$ and $b_n$, respectively, that
satisfy~\req{frame-prop} are referred to as frame bounds.
Notice that in a finite dimensional setting any family
that spans the whole space is a frame.

We further denote by $\Base_n \colon \R^{\I_n}  \to \R^{\Om_n}$ the operator that maps the signal $\signal_n \in \R^{\I_n}$ to the  analyzing coefficients with respect  to the given
frame,
 \begin{equation*}
\kl{\forall \om \in \Om_n}
\qquad
\kl{\Base_n\signal_n}\skl{\om} := \inner{\base_{\om}^{n}}{\signal_n}\,.
\end{equation*}
The mapping $\Base_n$ is  named the
\emph{analysis operator},
its adjoint   $\Base_n^\ast$ the  \emph{synthesis operator},
and $\Base_n^\ast\Base_n$  the \emph{frame operator}
corresponding to $\dict_n$.

The frame property \req{frame-prop} implies that the frame operator $\Base_n^\ast\Base_n\colon \R^{\I_n} \to \R^{\I_n}$ is an invertible
linear mapping.  Hence, for any $\om \in \Om_n$, the elements
\begin{equation*}
\dbase_{\om}^{n} := \kl{\Base_n^\ast\Base_n}^{-1}
\base_{\om}^{n}
\end{equation*}
are well defined  and the family $\mkl{\dbase_{\om}^{n}: \om\in\Om_n}$
is again a frame of $\R^{\I_n}$.
It is called the \emph{dual frame}
and has frame bounds  $1/b_n \leq  1/a_n$.

Finally, we denote by  $\Base_n^+ :=  \mkl{\Base_n^\ast\Base_n}^{-1}\Base_n^\ast$  the pseudoinverse of the analysis operator
$\Base_n$.
Due to linearity and the definitions of the pseudoinverse and the dual frame elements, we have the identities
\begin{equation} \label{eq:inv}
\kl{\forall \signal_n \in \R^{\I_n}}
\qquad
\signal_n
=
\Base_n^+ \Base_n \signal_n
=
\sum_{\om\in \Om_n}
\inner{\base_{\om}^{n}}{\signal_n}
\dbase_{\om}^{n}
  \,.
\end{equation}
In particular, the mapping $\Base_n^+$ is the synthesis  operator corresponding to the dual frame.
Equation~\req{inv} provides  a simple  representation of
the given signal in terms of its  analyzing coefficients.
This serves as basis of thresholding estimators defined and studied  in the following subsection.
For further details  on frames see, for example,  \cite{Chr08,Mal09}.

\begin{remark}[Thresholding in a subspace]\label{rem:frame}
It is not essential at all, that $\dict_n$
is a frame of the whole image space $\R^{\I_n}$.
In fact, in typical  thresholding  applications, such as  in
wavelet denoising, the space $\R^{\I_n}$
naturally  decomposes  into a low resolution space
having  small fixed dimension  and a detail space
having  large dimension that increases with $n$.
The  soft-thresholding procedure is then only applied to the signal part
in the detail space and hence it is sufficient to assume that $\dict_n$
is a frame therein.  In order to avoid unessential technical complication we present our results for the case of  frames of the whole image space.
In the concrete  applications presented in Section \ref{sec:ez} the thresholding will indeed only be performed in some subspace;   all results carry over to such a situation in a straightforward manner.\end{remark}

\subsection{Thresholding Estimation}
 \label{sec:soft}

By  applying  $\Base_n$ to both sides of~\req{problem},
the original denoising problem  in the signal space $\R^{\I_n}$
is transferred into the  denoising problem
\begin{equation} \label{eq:problem2}
	\Data_n \skl{\om}
	=
	\Para_n \skl{\om} +  \kl{\Base_n\noise_n}  \skl{\om} \,,
	\qquad \text{ for } \om  \in  \Om_n   \,,
\end{equation}
in the possibly higher dimensional coefficient space
$\R^{\Om_n}$.  Here and in the following we denote by
\begin{equation} \label{eq:DataPara}
\Data_n \skl{\om}
:= \inner{\base_{\om}^{n}}{\data_n} \,
\quad \text{ and } \quad
\Para_n \skl{\om}
:= \inner{\base_{\om}^{n}}{\signal_n} \,,
\end{equation}
the coefficients of the data $\data_n$ and  the signal
$\signal_n$  with respect to the given frame.
The  following  elementary Lemma \ref{lem:covariance}
states that the noise term in~\req{problem2}
is again a centered normal vector   but  has possibly non-vanishing
covariances. Indeed it implies that the entries of $\Base_n\noise_n$ are not uncorrelated and hence not  independent, unless $\dict_n$  is an
orthogonal basis.

\begin{lemma}[Covariance matrix]\label{lem:covariance}
Let $\noise_n$ be a random vector in the image space $\R^{\I_n}$
with independent $N\skl{0,\sigma^2}$-distributed
entries. Then $\Base_n\noise_n$ is a  centered normal vector in $\R^{\Om_n}$ and the  covariance matrix of
$\Base_n\noise_n$ has
entries
$\cov \mkl{ \sl{\Base_n\noise_n}\skl{\om},
\Base_n\noise_n \skl{\om'}} = \sigma^2
\minner{\base_{\om}^{n}}{\base_{\om'}^{n}}$.
\end{lemma}

\begin{proof}
As the sum  of normal random variables with zero
mean,  the random variables  $\skl{\Base_n\noise_n}\skl{\om} = \sum_{k\in \I_n} \base_{\om}^{n}\kl{k} \noise_n \kl{k}$ are  again normally distributed with zero mean.  In particular, we have   $\cov \mkl{ \Base_n\noise_n \skl{\om},  \Base_n\noise_n \skl{\om'}} = \ew \mkl{  \Base_n\noise_n \skl{\om}  \Base_n\noise_n \skl{\om'} }$. Hence the claim follows from the linearity of the expectation value and the  independence of  $\noise_n\kl{k}$.
\end{proof}

Recall the soft-thresholding function
$\soft \kl{y, T_n} = \operatorname{sign}\kl{y} \kl{\abs{y}-T_n}_+$  defined by Equation \req{soft}.
The thresholding estimators we consider apply
$\soft \mkl{\edot, T_n}$ to each  coefficient of $\Data_n$ in \req{problem2} to define an estimator for the
parameter $\Para_n$.
In order to get an estimate for the signal $\signal_n$ one must map the
coefficient estimate  back to  the original signal
domain. This is usually implemented by applying   the dual synthesis operator
(compare with Equation~\req{inv}).

\begin{definition}[Frame thresholding]
\label{def:soft-frame}
Consider the data models~\req{problem}
and~\req{problem2}  and let $T_n >0$ be a given thresholding parameter.
\begin{enumerate}[label=(\alph*),topsep=0.0em]
\item
The   soft-thresholding estimator for
$\Para_n \in \R^{\Om_n}$ using the threshold
$T_n$ is defined by
\begin{equation}\label{eq:est-soft-para}
\hat\Para_n
=
\softop \kl{\Data_n, T_n}
:=
\kl{\soft \skl{\Data_n \skl{\om}, T_n}: \om \in \Om_n}
\in \R^{\Om_n} \,.
\end{equation}

\item
The   soft-thresholding  estimator for  $\signal_n$
with respect to the frame $\dict_n$ using the threshold $T_n$
is defined by
\begin{equation} \label{eq:est-soft}
\hat \signal_n
=
\Base_n^+ \circ \softop \kl{\Base_n \data_n, T_n}
=
\sum_{\om \in \Om_n}  \soft \kl{\sinner{\base_{\om}^{n}}{\data_n}, T_n} \dbase_{\om}^n \,.
\end{equation}
\end{enumerate}
Hence  the frame soft-thresholding estimator $\hat\signal_n$ is simply the  composition of  analysis with $\Base_n$, component-wise thresholding, and    dual synthesis with $\Base_n^+$.\end{definition}

If  $\dict_n$ is an overcomplete
frame, then $\Base_n$ has infinitely many
left-inverses,  and the pseudoinverse used
in Definition~\ref{def:soft-frame} is a particular one.
In principle one could use other left inverses
for defining the  soft  thresholding   estimator~\req{est-soft}.
Since, in general, $\softop \kl{\Data_n, T_n} \not \in  \range\mkl{\Base_n }$ is  outside the range of
$\Base_n$, the use  of a  different left inverse will result in a different estimator.
The  special choice  $\Base_n^+$ has the advantage that for many frames used in practical applications, the dual synthesis operator is known explicitly  and, more importantly, that fast algorithms  are available for its computation (typically algorithms using  only $\mathcal O\mkl{\sabs{\I_n}\log \sabs{\I_n}}$ or even
$\mathcal O\mkl{\sabs{\I_n}}$ floating point operations \cite{Mal09}).

\begin{remark}[Thresholding variations] \label{rem:other-thresholding}
Instead of the \emph{soft} thresholding function
$\soft \mkl{\edot, T_n}$ several other nonlinear thresholding methods have been proposed and used. Prominent examples are the hard thresholding function $z \mapsto z \chi_{\set{\abs z \geq T_n}}$ and the nonnegative garrote  $z \mapsto  z \max\sset{1-T_n^2/z^2,0}$  of~\cite{Bre95,Gao98}.
Strictly taken, the smoothness estimates derived in
Section~\ref{sec:smoothness} only hold for
thresholding functions $F\mkl{\edot,T_n}$ satisfying the shrinkage property
$\sabs{F\mkl{y
\pm T_n,T_n}} \leq  \sabs{y}$ for all $y \in \R$. This property is, for example,  not satisfied by  the nonnegative garrote. In this case, however, similar estimates may be derived under additional assumptions on the signals of interest.
Other prominent denoising techniques are based on
block-thresholding (see, for example,~\cite{Cai99,CaiZho09,CheFadSta10a,HalKerPic99}).
In this case, the derivation of sharp smoothness estimates requires extreme value results for dependent
$\chi^2$-distributed random variables
(with more than one degree of freedom).
Such an extreme value analysis  seems possible but is
beyond the scope of this paper.
Our given results can be seen as the first step towards a
more general theory.
\end{remark}

\begin{remark}[Multiple selections]\label{rem:dict}
Be aware,  that we allow certain elements
$\base_{\om}^{n}$ to be contained  more
than once in the frame $\dict_n$.
Hence we may have $\sabs{\sset{\base_{\om}^{n}: \om \in\Om_n}} < \abs{\Om_n}$.
Such multiple selections often arises naturally for  frames
that are the union of several bases having some elements in
common.
A standard example is the wavelet cycle spinning
procedure  of~\cite{CoiDon95}, where the underlying
frame is the union of several shifted orthonormal wavelet bases
(see Section~\ref{sec:cs}).
Multiple selections of frame elements   also affect the pseudoinverse and finally the soft-thresholding estimator.
Hence, if
$\mkl{\base_{\om}^{n} : \om\in \Om_n}$ and
$\mkl{\psi_\la^n : \la \in \Lambda_n}$ denote   two
frames composed  by  the same frame elements,
$\mset{\base_{\om}^{n}:\om\in \Om_n}
=
\mset{\psi_\la^n: \la\in \Lambda_n}$,
but having different cardinalities $\abs{\Om_n} \neq \sabs{\Lambda_n}$, then the  soft-thresholding estimators corresponding to these frames differ from each other.\end{remark}

\subsection{Rationale Behind Thresholding Estimation}

We conclude this section by commenting on the
underlying rationale behind  thresholding estimation and
situations where it is  expected to produce good results.

The basic assumption underlying thresholding estimation is that the frame
operator separates the  data into large coefficients
due to the  signal and  small coefficients   mainly due to the noise.
For additive noise models   $\data_n = \signal_n + \noise_n$
both issues can be studied separately. In this case,   one requires that for some  threshold $T_n$
(which  finally judges between signal  and  noise)
the following two conditions are satisfied:

\begin{enumerate}[label=(\arabic*), topsep=0.0em]
\item \label{it:principle1}
\textbf{Coherence between signal and  frame:}
The signal $\signal_n$ is well represented by few large
coefficients having $\sabs{\sinner{\base_{\om}^{n}}{\signal_n}} >T_n$.

\item \label{it:principle2}
\textbf{Incoherence between noise  and frame:}
With high probability, all noise coefficients with respect to the frame
$\dict_n$ satisfy
$\sabs{\sinner{\base_{\om}^{n}}{ \noise_n}}
\leq  T_n$.
\end{enumerate}

In the following sections we shall see, that Item~\ref{it:principle2} can be analyzed in a unified way for asymptotically stable  frames.
Item~\ref{it:principle1}, however, is more an approximation issue
rather than an estimation issue. Given a frame, it, of course, cannot  be satisfied for every
$\signal_n \in \R^{n}$. The choice of a `good frame' depends on the specific application at hand and in particular on the type of signals that  are expected.
The better the signals of interest are represented by a few but large frame coefficients, the better the denoising result will be. The richer the analyzing family is,
the more signals can be expected  to be recovered properly.
The price to pay must be, of course, a higher computational cost.

The following two simple examples demonstrate how
the use of redundant frames may significantly improve the performance of the thresholding estimation.

\begin{figure}[htb!]\centering
\includegraphics[width=\textwidth]{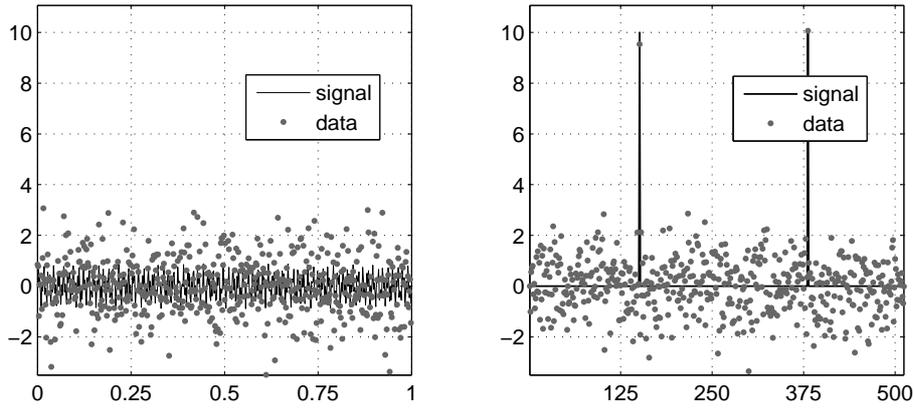}
\caption{Left: Signal $\signal_n$ (superposition of two  sine waves)
and data $\data_n = \signal_n + \noise_n$ from Example~\ref{ex:sine1}. Right: Coefficients  of the signal and the data with respect to the
sine basis.\label{fig:sine1}}
\end{figure}

\begin{example}[Thresholding in the sine basis] \label{ex:sine1}
We consider the discrete signal $\signal_n \in \R^{n}$ defined by  $\signal_n\kl{k} = 5\sqrt{2}/16$ $\sin\mkl{\pi \om_1 k /n} + 5\sqrt{2}/16\,\sin\mkl{\pi\om_2 k /n}$, which is a  superposition of two sine waves having  frequencies $\om_1 =150$ and $\om_2 =380$, respectively,  and amplitudes $5\sqrt{2}/16 \simeq 0.45$.
The left image  in Figure~\ref{fig:sine1} shows the  signal $\signal_n$ and the noisy data $\data_n = \signal_n + \noise_n$ obtained  by  adding Gaussian  white noise of variance equal to one to the signal.
Apparently,  there seems little  hope to recover  $\signal_n$ from the data $\data_n$ in the original signal domain.
Almost like a miracle, the  situation changes
drastically  after computing the  coefficients with respect to the sine basis $\mkl{n^{-1/2} \sin \skl{\pi\om k /n}: \om=1, \dots, n}$.
Now, the signal and the noise are clearly separated as can be seen from the right image  in Figure~\ref{fig:sine1}.
Obviously we will get an almost perfect reconstruction
by simply removing all coefficients below a proper
threshold.
\end{example}

\begin{figure}[htb!]\centering
\includegraphics[width=\textwidth]{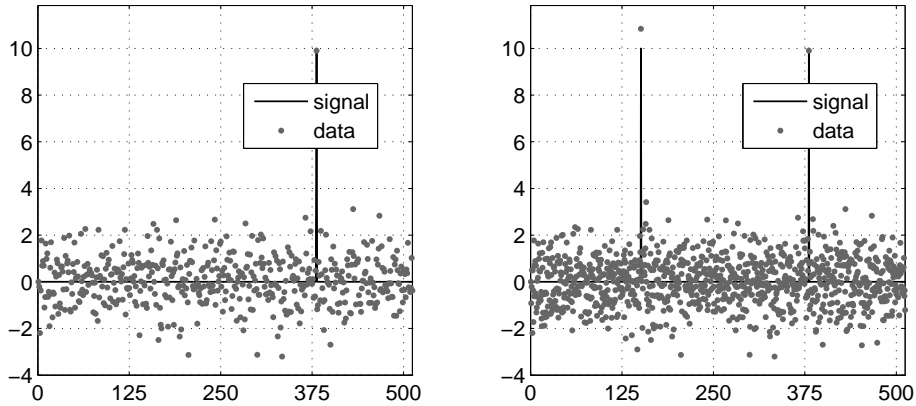}
\caption{Left: Coefficients  of the signal $\signal_n'$ and the data
$\data_n' = \signal_n' + \noise_n$ from Example~\ref{ex:sine2} with respect the sine basis. Right: Coefficients of the same signal and data  with respect to the two times oversampled sine frame. \label{fig:sine2}}
\end{figure}

\begin{figure}[htb!]\centering
\includegraphics[width=\textwidth]{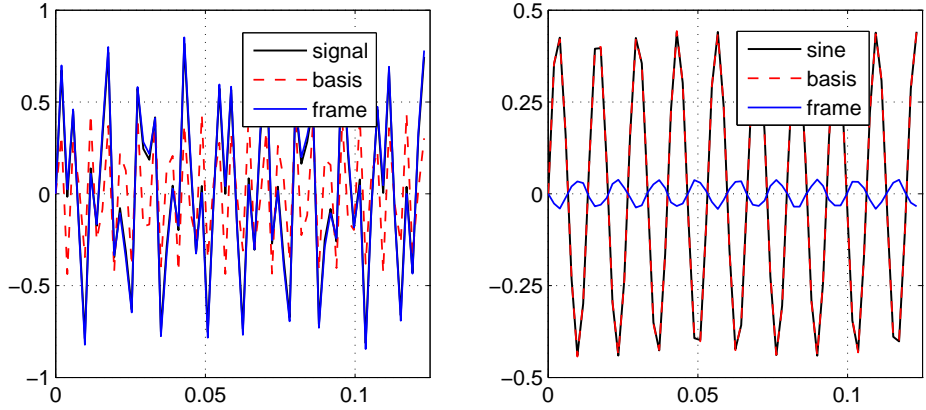}
\caption{Left: Signal $\signal_n'$,
and the reconstructions from the data $\data_n' = \signal_n' + \noise_n$  by soft-thresholding  in the sine basis and in the  overcomplete sine frame, respectively.
Only the first $64$ components  are plotted. Right: Sine wave
$5\sqrt{2}/16 \, \sin\mkl{\pi \om_1' k /n}$  and residuals of the
two reconstructions. As can be seen, thresholding in the sine frame almost perfectly recovers the signal $\signal_n$, whereas the result of thresholding in the sine basis is useless
(the residual is almost equal to the displayed  sine wave of frequency $\om_1'$.)
\label{fig:sine2-rec}}
\end{figure}

\begin{example}[Thresholding in a redundant sine frame] \label{ex:sine2}
The signal  in Example~\ref{ex:sine1} is  a  combination
of sine waves with integer frequencies covered by the sine frame.   However, in practical  application the signal may
also have non-integer frequencies.
In order to investigate this issue, we now consider the
signal $\signal_n' \kl{k} = 5\sqrt{2}/16 \, \sin\mkl{\pi \om_1' k /n} + 5\sqrt{2}/16\,\sin\mkl{\pi\om_2 k /n}$ having frequencies $\om_1' =150.5$ and
$\om_2 =380$
(hence $\om_1'$ is a slight perturbation of  the frequency $ \om_1$ considered in Example~\ref{ex:sine1}).
The new signal $\signal_n'$ is not a  sparse linear combination of elements of the sine basis. As a matter of fact, the energy of the first sine wave  is spread over many  coefficients  and thus      submerges in the noise.
Indeed, as can be seen from the left image in  Figure~\ref{fig:sine2}, the low frequency coefficient  disappears.
However, by taking the two times redundant frame
$\mkl{n^{-1/2}\sin \kl{\pi\om k /n} : \om=\set{1/2,1, \dots, n}}$ instead  of the sine basis, the coefficient due to frequency $\om_1'$ appears
again in the transformed domain.
Moreover, as can be seen from Figure~\ref{fig:sine2-rec} the reconstruction by thresholding the coefficients with respect to the overcomplete sine frame is almost perfect, whereas the reconstruction by thresholding the basis coefficients is useless.
\end{example}

In Examples~\ref{ex:sine1} and~\ref{ex:sine2}
the  threshold choice is not a very delicate issue
since the signal and the noise are separated very
clearly in the transformed domain. Indeed as  can be seen from the right plots in Figures~\ref{fig:sine1} and ~\ref{fig:sine2}
there is a quite wide  range of thresholds that would yield an almost  noise free estimate close to the original signal.
However,  if the signal also contains important
coefficients of moderate size, then the choice of a good
threshold is crucial and difficult.
This is typically the case for image denoising using wavelets  or curvelet frames:
Natural  images are approximately sparse in these frames but almost never
strictly sparse. The   particular  threshold choice now will always be a tradeoff
between noise removal and signal representation and becomes a delicate
issue.
In order to develop rationale threshold choices, a precise  understanding
of the distribution of $\sabs{\sinner{\base_{\om}^{n}}{\noise_n}}$ is helpful.
This is the subject of our following considerations.

\section{Extreme Value Analysis of Frame Thresholding}
\label{sec:evd-frame}

Now we turn back to the denoising  problem~\req{problem}.
After  application of the analysis operator $\Base_n$ corresponding to  the normalized frame $\dict = \mkl{\base_{\om}^{n}: \om \in \Om_n}$ our aim is to estimate
the vector $\Para_n \in \R^{\Om_n} $ from given noisy coefficients (compare with  Equation~\req{problem2})
\begin{equation*}
	\Data_n \skl{\om}
	=
	\Para_n \skl{\om} +     \kl{\Base_n\noise_n}  \skl{\om} \,,
	\qquad \text{ for } \om  \in  \Om_n   \,.
\end{equation*}
Here $\Base_n\noise_n$ is the transformed noise vector which is normally distributed, has zero mean and covariance matrix
$\kappa_n \skl{\om,\om'} =
\sigma^2 \sinner{\base_{\om}^{n}}{\base_{\om'}^{n}}$; see
Lemma~\ref{lem:covariance}.
In this section we shall analyze in detail the
component-wise soft-thresholding estimator
$\hat\Para_n = \softop  \mkl{\Data_n, T_n}$
defined by~\req{est-soft-para}.  We will start by   computing the  extreme value distribution of $\Base_n\noise_n$ claimed in Theorem~\ref{thm:main}. Based on the  limiting law will then introduce extreme value thresholding techniques
that will be shown to provide asymptotically  sharp
confidence regions.

\subsection{Proof of Theorem~\ref{thm:main}}
\label{sec:thm:main}

The main aim of this subsection is to verify Theorem~\ref{thm:main}, which states that the distribution of the  maxima of the noise coefficients
$\Base_n\noise_n  \skl{\om}$ are of Gumbel type with explicitly given normalization constants.
The proof of  Theorem~\ref{thm:main} will be a
consequence of  Lemmas~\ref{lem:main}  and~\ref{lem:bessel-decay}
to be derived in the following. The main Lemma~\ref{lem:main}  relies itself  on
a new extreme value result established in Section~\ref{ap:evd}.

\begin{lemma}\label{lem:main}
Let $\mkl{\Noise_n}_{n \in \N}$ be a sequence of normal  random vectors in $\R^{\Om_n}$ with covariance matrices $\kappa_{n}$ having ones in the diagonal. Assume additionally, that the following holds:
\begin{enumerate}[label=(\roman*),topsep=0em]
\item \label{it:assi1}
For  every $\delta \in \kl{0,1}$,
$\abs{\mset{\skl{\om,\om'} \in \Om_n^2 :
\sabs{\kappa_{n}  \skl{\om,\om'} } \geq  \delta }}
= \mathcal O\kl{\abs{\Om_n}}$ as $n\to \infty$.

\item \label{it:assi2}
For some $\rho \in\kl{0,1}$,
$\abs{\mset{\skl{\om,\om'} \in \Om_n^2:
\sabs{\kappa_{n}  \skl{\om,\om'} } \geq  \rho }}
= o\kl{\abs{\Om_n} / \sqrt{\log \abs{\Om_n}} }$ as $n\to \infty$.

\item \label{it:assi3}
$B := \sup \set{
 \sum_{\om' \in \Om_n}  \sabs{\kappa_{n} \skl{\om, \om'}}^2  : n \in \N \text{ and }\om \in \Om_n} < \infty$.
\end{enumerate}
Then, $\norm{\Noise_n}_\infty$ is of Gumbel type (see Definition~\ref{def:gumbeltype}) with  normalization constants  $a\mkl{\chi,\abs{\Om_n}} $ and
$b\mkl{\chi,\abs{\Om_n}}$ defined by  \req{a-abs} and \req{b-abs}.
\end{lemma}

\begin{proof}
Let  $\mkl{\Noise_n}_{n \in \N}$ be a sequence of normal  random vectors  satisfying Conditions~\ref{it:assi1}--\ref{it:assi3}.
According to  Theorem~\ref{thm:abs} it is sufficient to show that
\begin{equation*}
R_{n}
 :=
\sum_{\om\neq \om'}
\mabs{\kappa_{n} \skl{\om,\om'}}
 \,
\kl{ \frac{\log \abs{\Om_n}}{\abs{\Om_n}^2}}^{1/\kl{1+\sabs{\kappa_{n}\skl{\om,\om'}}}}
\to 0
\, \qquad \text{ as  }   n \to \infty \,.
\end{equation*}
This will be done by splitting the sum $R_{n}$ into three parts and showing  that  each of them  tends to zero as $n \to \infty$.
For that purpose, let $\delta \in \kl{0, 1/3}$ be any small
number, let $\rho \in \kl{0, 1}$  be as  in Condition~\ref{it:assi2}
and define
\begin{align*}
\Lambda_n\kl{1}
&:=
\mset{ \skl{\om,\om'}  \in \Om_n^2: \om \neq \om' \text{  and }
\sabs{\kappa_{n} \skl{\om,\om'}} \geq  \rho } \,,
\\
\Lambda_n\kl{2}
&:=
\mset{ \skl{\om,\om'}  \in \Om_n^2:
\delta \leq \sabs{\kappa_{n} \skl{\om,\om'}}  < \rho }
\,,
\\
\Lambda_n\kl{3}
&:=
\mset{ \skl{\om,\om'} \in \Om_n^2:
\sabs{\kappa_{n} \skl{\om,\om'}}  < \delta }
\,.
\end{align*}
We further write $R_{n} = R_{n}\kl{1} +  R_{n}\kl{2} + R_{n}\kl{3} $
with
\begin{equation*}
R_{n}\kl{i}
=
\sum_{\kl{\om, \om'} \in \Lambda_n\kl{i}}
\abs{\kappa_{n} \skl{\om,\om'}}
 \,
\kl{ \frac{\log \abs{\Om_n}}{\abs{\Om_n}^2}}^{1/\kl{1+\sabs{\kappa_{n}\skl{\om,\om'}}}}
\quad \text{ for } i \in \sset{1,2,3} \,.
\end{equation*}
It remains to verify that any of the terms
$R_{n}\kl{i}$  converges  to zero as  $n \to \infty$.

 \begin{itemize}[topsep=0.0em]
 \item
Since any $\Noise_n$ is a normal random vector  with zero mean and unit variance, we have
 $\sabs{\kappa_{n}\skl{\om,\om'}} \leq 1$ for any index pair
 $\skl{\om,\om'} \in \Om_n^2$, which yields the inequality
 $R_{n}\kl{1} \leq  \xsabs{\Lambda_n\kl{1}} \,
 \sqrt{\log\abs{\Om_n}} / \xsabs{\Om_n}$.
 By Condition~\ref{it:assi2}  we  have
 $\sabs{\Lambda_n\kl{1}} = o \kl{\abs{\Om_n}/\sqrt{\log\abs{\Om_n}} }$ which  shows that
 $R_{n}\kl{1} \to 0$ as  $n \to \infty$.

 \item
To estimate the second sum $R_{n}\kl{2}$,    recall that by definition of the set $\Lambda_n\kl{2}$, we have  $\sabs{\kappa_{n}\skl{\om,\om'}} \leq \rho$ for any pair of indices  $\skl{\om,\om'} \in \Lambda_n\kl{2}$.   Moreover, recall that by Condition~\ref{it:assi1} we further have    $\sabs{\Lambda_n\kl{2}} = \mathcal O
\kl{\abs{\Om_n}}$.
Hence we obtain
 \begin{equation*}
 R_{n}\kl{2}
\leq
\mabs{\Lambda_n\kl{2}} \,
\kl{\frac{\log\abs{\Om_n}}{\abs{\Om_n}^2}}^{1/\kl{1+\rho}}
=
\kl{\log\abs{\Om_n}}^{1/\kl{1+\rho}} \,
\mathcal O \kl{ \abs{\Om_n}^{1- 2/\kl{1+\rho}}} \,.
 \end{equation*}
Since by assumption $\rho < 1$, the inequality
$1 - 2/\kl{1+\rho} <  0 $ holds which implies that we have
$R_{n}\kl{2} \to 0$ as  $n \to \infty$.

\item
It remains to estimate the final sum  $R_{n}\kl{3}$.
The Cauchy-Schwarz inequality, Condition~\ref{it:assi3}, and the estimate $\sabs{\kappa_{n}\skl{\om,\om'}} \leq \delta$
yield
\begin{multline*}
R_{n}\kl{3}^2
 \leq
 \sum_{\kl{\om, \om'} \in \Lambda_n\kl{3}}
\abs{\kappa_n\skl{\om,\om'}}^2
\sum_{\kl{\om, \om'} \in \Lambda_n\kl{3}}
 \kl{\frac{\log\abs{\Om_n}}{\abs{\Om_n}^2}}^{2/\kl{1+\delta}}
\\
 \leq
B \, \abs{\Om_n}
\;
\kl{\frac{\log\abs{\Om_n}}{\abs{\Om_n}^2}}^{2/\kl{1+\delta}}
 \abs{\Om_n}^2
=
\kl{\log\abs{\Om_n}}^{2/\kl{1+\delta}}
\mathcal O\kl{ \abs{\Om_n}^{3  - 4 /\kl{1+\delta}}}
\,.
\end{multline*}
Now, by assumption  the inequality  $\delta< 1/3$ holds and hence  we have $4/\kl{1+\delta} > 3$.
This implies that  also $R_{n}\kl{3}$ tends to zero
as $n \to \infty$.
\end{itemize}

In summary, we have verified that $R_{n}\kl{i} \to 0$ as  $n \to \infty$ for every $i \in \set{1,2,3}$. Hence  their sum $R_{n}$ converges to zero, too.
The claimed distributional convergence results now
follows from Theorem~\ref{thm:abs}  and concludes the proof.
\end{proof}

We next state a simple auxiliary Lemma that  bounds the
number of inner products  $\minner{\phi_{\om}^n}{\base_{\om'}^n}$  being bounded away  from zero.

\begin{lemma}\label{lem:bessel-decay}
For any $n$ let $\mkl{\base_{\om}^{n}: \om\in\Om_n}$ be a
family of normalized vectors in $\R^{\I_n}$, such that the upper frame bounds $b_n$ are uniformly bounded.
Then, for every $\delta > 0 $, we have
 \begin{equation} \label{eq:bessel-decay}
\abs{\sset{\skl{\om,\om'}\in \Om_n^2:
\sabs{\sinner{\phi_{\om}^n}{\base_{\om'}^n}} \geq \delta } }
= \mathcal O\kl{ \abs{\Om_n} }  \;.
\end{equation}
\end{lemma}

\begin{proof}
To verify~\req{bessel-decay} it is sufficient to find,
for every given $\delta >0$, some constant $K \in \N$
such that
 \begin{equation} \label{eq:bessel-decay1}
\skl{\forall n  \in \N}
\skl{\forall \om  \in \Omega_n}
\qquad
\abs{\sset{\om'\in \Omega_n:  \sabs{\sinner{\phi_{\om}^n}{\base_{\om'}^n}} \geq \delta } } \leq K \;.
\end{equation}
Indeed,  if~\req{bessel-decay1} holds then summing over all
$\om \in \Om_n$ yields~\req{bessel-decay}.

To show \req{bessel-decay1} we assume to the contrary that
there is some  $\delta > 0$ such that for all $m\in \N$ there exists
some $n\kl{m} \in \N$ and some $\om \in \Omega_{n\kl{m}}$
such that the set $\Lambda_m = \sset{\om'\in \Omega_{n\skl{m}}:  \sabs{\sinner{\base^{n\skl{m}}_{\om}}{\base_{\om'}^{n\kl{m}}}} \geq \delta }$ contains  more then  $m$ elements.
By Assumption we have the equality $\snorm{\base_{\om}^{n\kl{m}}} =1 $ for all $\om\in \Om_n$. Together with
Assumption~\ref{it:frame3} this implies
\begin{equation*}
B = B \, \mnorm{\base_{\om}^{n\kl{m}}}^2
\geq
\sum_{\om'\in \Omega_{n\skl{m}}} \mabs{\minner{\base_{\om}^{n\skl{m}}}{\base_{\om'}^{n\kl{m}}}}^2
\geq
\sum_{\om'\in \Lambda_M}
\mabs{\minner{\base^{n\skl{m}}_{\om}}{\base_{\om'}^{n\kl{m}}}}^2
\geq
m \delta  \,.
\end{equation*}
Since the last estimate should hold for all $m\in \N$ and we have $B<\infty$
by assumption, this  obviously gives is a contradiction.
\end{proof}

\begin{proof1}
Theorem~\ref{thm:main} is now an immediate
consequence of the above results:  Lemma~\ref{lem:covariance}  and
Lemma~\ref{lem:bessel-decay} show that the sequence of  normalized frame coefficients
$\skl{\Base_n\noise_n/\sigma}_{n\in \N}$ satisfies
Conditions~\ref{it:assi1}--\ref{it:assi3} of Lemma~\ref{lem:main}.
Hence Lemma~\ref{lem:main} applied to  the random vectors
$\Noise_n  = \Base_n\noise_n / \sigma $
shows the assertion.
\end{proof1}

We conclude this subsection by stating an  extreme value result for the  maximum without the absolute values.
Although we do not need this result further, the distributional limit is interesting in its own.

\begin{theorem}[Limiting Gumbel law without absolute values]\label{thm:mainb}
Assume that the frames $\dict_n$ are asymptotically stable and let $\skl{\noise_n}_{n\in\N}$ be a sequence of random vectors in $\R^{\I_n}$ having independent $N\skl{0,1}$-distributed entries.
Then, the random sequence of  the maxima  $ \max\kl{\Base_n\noise_n} :=
\max  \sset{ \sinner{\base_\om^n}{\noise_n}: \om \in \Om_n}$ is of Gumbel type with normalization constants
\begin{align}\label{eq:am}
a\mkl{N,\abs{\Om_n}}  &:=  \frac{1}{\sqrt{2 \log \abs{\Om_n}}} \,,
\\ \label{eq:bm}
b\mkl{N,\abs{\Om_n}}  &:= \sqrt{2 \log \abs{\Om_n}} -
 \frac{ \log \log \abs{\Om_n} +   \log \kl{4 \pi}} {2 \sqrt{2\log \abs{\Om_n}}} \,.
\end{align}
\end{theorem}

\begin{proof}
The proof  is analogous to the proof of Theorem~\ref{thm:main} and uses the extreme value result of Theorem~\ref{thm:normal} for dependent normal random vectors instead of the one of Theorem~\ref{thm:abs} for  absolute values of dependent normal random vectors.
\end{proof}

\subsection{Universal Threshold:
Qualitative  Denoising Property}
\label{sec:denoising}

In the case that  the family $\dict_n = \mkl{\base_{\om}^{n}: \om \in \Om_n}$ is an orthonormal basis
it is well known that the thresholding sequence
$T_n  =  \sigma \sqrt{2 \log \abs{\Om_n}}$  satisfies the asymptotic denoising property
(see, for example,  \cite{Don95,Joh11,Mal09} and also Section~\ref{sec:wave-soft} in the introduction), that is,
\begin{equation}  \label{eq:noise-remove2}
	\lim_{n \to \infty}
	\wk \set{ \snorm{\Base_n \noise_n}_\infty
	\leq   T_n }
	= 1 \,.
\end{equation}
Equation~\req{noise-remove2} implies that the
estimates  obtained with the   threshold
$\sigma \sqrt{2 \log \abs{\Om_n}}$  are,
with probabilities tending  to one,
at least as smooth as $\signal_n$.
Hence the relation \req{noise-remove2}
is often used as  theoretical justification for using the
universal threshold choice $\sigma \sqrt{2 \log \abs{\Om_n}}$ originally proposed by  Donoho and Johnstone (see \cite{Don95,DonJoh94}).
 The following Proposition~\ref{prop:denoise} states that
the same denoising property indeed holds  true for
any normalized frame. Actually it proves much more:
First, we verify~\req{noise-remove2} for a wide class of thresholds including the Donoho-Johnstone  threshold.
Second, we show that this class in fact includes all
thresholds that  satisfy the denoising
property~\req{noise-remove2} -- provided
that  the frames are asymptotically stable.
Our results can be seen as a generalization and a refinement of~\cite[Theorem~4.1]{Don95} from the basis case to the possibly redundant frame case.

\begin{proposition}[Thresholds yielding the denoising property]\label{prop:denoise}
Assume that $\dict_n$ are frames of $\R^{\I_n}$ having  normalized frame elements  and analysis operators $\Base_n$, and let $\skl{\noise_n}_{n\in\N}$ be a sequence of noise vectors in $\R^{\I_n}$ with independent $N\skl{0,\sigma^2}$-distributed entries.
\begin{enumerate}[label=(\alph*),topsep=0.0em]
\item\label{it:denoise1}
If $\skl{\dict_n}_{n \in \N}$ is asymptotically stable, then  a sequence  $\mkl{T_n}_{n \in \N}$
of thresholds satisfies~\req{noise-remove2} if and only
if it has  the form
\begin{equation} \label{eq:denoise-sharp}
T_n  :=
\sigma \sqrt{2 \log \abs{\Om_n}}
+
\sigma \, \frac{2 z_n  -  \log\log \abs{\Om_n}- \log \pi}
{2 \sqrt{2 \log \abs{\Om_n}}}
\quad \text{ with  }\quad
\lim_{n \to \infty} z_n   =  \infty \,.
\end{equation}

\item\label{it:denoise2}
If  $\skl{\dict_n}_{n \in \N}$ is  not  necessarily asymptotically
stable, then still any sequence
$\mkl{T_n}_{n \in \N} \subset \kl{0,\infty}$
of the  form~\req{denoise-sharp} satisfies
the asymptotic denoising property~\req{noise-remove2}.
\end{enumerate}
\end{proposition}

\begin{proof}\mbox{}
\ref{it:denoise1}
Theorem~\ref{thm:main} immediately  implies that
a sequence $\skl{T_n}_{n\in\N}$
satisfies~\req{noise-remove2} if and only  if it has the form
\begin{equation*}
T_n
 =
\sigma \sqrt{2 \log \abs{\Om_n}} +
\sigma  \, \frac{2 z_n  -   \log\log
\abs{\Om_n} - \log \pi}{ 2\sqrt{2 \log \abs{\Om_n}}}
\end{equation*}
for  some sequence $\skl{z_n}_{n \in \N}$ with $z_n \to \infty$.
%Hence the claim  follows by taking $\qq_n :=  2 z_n - \log \pi$.

\ref{it:denoise2}
Now let   $\dict_n$  be any sequence  of frames that is not necessarily asymptotically stable.
Further, let $\eta_n$ be a sequence of  random vectors  with independent $N\skl{0,\sigma^2}$-distributed entries.
Since $\Base_n \noise_n$ is a random vector with possibly dependent
$N\skl{0,\sigma^2}$-distributed entries,   Lemma~\ref{lem:sidak}
implies that
\begin{equation*}
	\wk
	\set{ \norm{\Base_n \noise_n}_\infty  \leq
	T_n }
	\geq
	\wk
	\set{ \norm{\eta_n}_\infty  \leq
	T_n} \,.
\end{equation*}	
By Item~\ref{it:denoise1}
we already know that $\wk \{ \snorm{\eta_n }_\infty \leq  T_n \} \to 1$ as $n \to \infty$,
for any sequence of thresholds satisfying \req{denoise-sharp}, and hence the same must  hold  true for
$\wk\{ \snorm{\Base_n \noise_n}_\infty  \leq
 T_n\}$.
\end{proof}

According to Proposition~\ref{prop:denoise}, any sequence $\skl{z_n}_{n\in\N}$  with $\lim_{n \to \infty} z_n  =  \infty$   defines a sequence
of thresholds~\req{denoise-sharp} that satisfies the asymptotic denoising property.
In particular, by taking  $2 z_n = \log\log \abs{\Om_n} + \log \pi$ the thresholds
in \req{denoise-sharp} reduce to the universal threshold  $\sigma \sqrt{2 \log \abs{\Om_n}}$ of Donoho and Johnstone.
Proposition~\ref{prop:denoise}  further shows  that the asymptotic relation
$T_n \sim \sigma \sqrt{2 \log \abs{\Om_n}}$ alone is not
sufficient  for the denoising property~\req{noise-remove2} to hold and
that second order approximations have to be considered. One may call a
thresholding sequence $\skl{T_n}_n$ smaller than
$\skl{T_n'}_n$,  if $ \mkl{T_n' - T_n} \, T_n \to \infty$
for $n \to \infty$.
The  smaller the thresholding sequence
is taken, the  slower the convergence  of $\wk \mset{ \snorm{\Base_n \noise_n}_\infty \leq   T_n}$ will be,
and hence this just yields a different compromise between
noise reduction and signal approximation.

\subsection{Extreme Value Threshold:
Sharp Confidence Regions}
\label{sec:alphathresh}

For the following notice that the soft-thresholding estimate
$\hat\Para_n =  \softop \kl{\Data_n,T_n}$ with thresholding parameter $T_n $ is an element
of  the $\enorm{}_{\infty}$-ball
\begin{equation}\label{eq:conf-ball}
	\ball\mkl{\Data_n, T_n} :=
	\set{\bar\Para_n \in \R^{\Om_n} :
	 \snorm{\bar\Para_n - \Data_n}_\infty
	 \leq T_n}
\end{equation}
around the given data $\Data_n$.
Our aim is to select the thresholding value $T_n$
in such a way,  that $\ball\mkl{\Data_n, T_n}$ is an asymptotically  sharp  confidence  region corresponding to some  prescribed  significance level  $\alpha$,
in the sense that the probability that  we have $\Para_n \in \ball\mkl{\Data_n, T_n}$ tends to  $1-\alpha$ as
$n \to \infty$.
By definition, $\Para_n \in \ball\skl{\Data_n, T_n}$ if and only if $\snorm{\Para_n - \Data_n}_\infty \leq  T_n$.  The data model $\Data_n = \Para_n + \Base_n\noise_n$ thus implies  that
\begin{equation}\label{eq:confidence1}
\wk \set{
\Para_n \in \ball\mkl{\Data_n, T_n}
\,;
\forall \Para_n  \in \range \skl{\Base_n}
}
=
\wk\set{\norm{\Base_n\eps_n}_\infty \leq
T_n } \,.
\end{equation}
Here and in similar situations,
$\wk \set{ \Para_n \in \ball\mkl{\Data_n, T_n}
\,;
\forall \Para_n  \in \range \skl{\Base_n}
}$ denotes the   probability
of the intersection   of  all the events $\set{\Para_n \in \ball\mkl{\Data_n, T_n}}$
taken over all $\Para_n \in \range \skl{\Base_n}$.

Now assume that the frames are asymptotically stable.
Then Theorem~\ref{thm:main} states that the probabilities in Equation~\req{confidence1} with $T_n
= \sigma a\mkl{\chi, \abs{\Om_n}}z  +
\sigma  b\mkl{\chi, \abs{\Om_n}}$ tend to the Gumbel distribution $\exp \kl{-\exp\kl{-z}}$.
This suggests the following threshold choice
based on the quantiles of the limiting Gumbel
distribution.

\begin{definition}[Extreme value threshold] \label{def:alphathresh}
Let  $\skl{\alpha_n}_{n\in\N} \in \kl{0,1}$
be any sequence of significance levels,
denote by
$z\skl{\alpha_n}  =
- \log \log \mkl{1/\skl{1-\alpha_n}}$  the
$\alpha_n$-quantile of  the Gumbel distribution,
and set
\begin{equation}\label{eq:alphathresh}
T \kl{\alpha_n, \abs{\Om_n}}
:=
\sigma  \sqrt{2\log \abs{\Om_n}}
+
\sigma \,
\frac{2z\skl{\alpha_n} - \log\log \abs{\Om_n} - \log \pi}{2  \sqrt{2\log \abs{\Om_n}}}
\,.
\end{equation}
We then name   $T \mkl{\alpha_n, \abs{\Om_n}}$
the sequence of extreme value threshold (EVT)
corresponding to the significance levels  $\alpha_n$.
\end{definition}

The following Theorem~\ref{thm:confidence}
states that the EVTs defined by Equation~\req{alphathresh} indeed define asymptotically sharp confidence regions.  Actually it is mere a corollary of the extreme value  result derived in Theorem~\ref{thm:main}.

\begin{theorem}[Asymptotically sharp confidence regions]\label{thm:confidence}
Let $\skl{\dict_{n}}_{n \in \N}$ be a asymptotically stable family of  frames in $\R^{\I_n}$ and let $\skl{\alpha_n}_{n \in \N}$ be a sequence of numbers  in
$\kl{0,1}$  converging  to some $\alpha \in [0,1)$.
Then, with the extreme value thresholds $T \mkl{\alpha_n, \abs{\Om_n}}$ defined in Equation~\req{alphathresh}, we have
 \begin{equation}\label{eq:confidence}
\lim_{n \to \infty}
\wk\set{
\Para_n \in
\ball\mkl{\Data_n, T \mkl{\alpha_n, \abs{\Om_n}}}
\,;
\forall \Para_n  \in \range \skl{\Base_n}}
 = 1-\alpha  \,.
\end{equation}
Hence, the sets $\ball\mkl{\Data_n, T \mkl{\alpha_n, \abs{\Om_n}}}$ defined in  \req{conf-ball} are asymptotically sharp confidence regions with significance level $\alpha$.
\end{theorem}

\begin{proof}
According to \req{confidence1} it is sufficient to show that
$\wk\mset{\snorm{\Base_n\eps_n}_\infty \leq T \mkl{\alpha_n, \abs{\Om_n}}}
\to 1-\alpha$ as $n \to \infty$.
 Theorem~\ref{thm:main} and the definition  of the thresholds in~\req{alphathresh} imply that the probability of the event $\mset{\snorm{\Base_n\noise_n}_\infty \leq  T \mkl{\alpha_n, \abs{\Om_n}}}$ converges  to $\exp\mkl{-\exp\skl{-
z\skl{\alpha}}}$ as $n \to \infty$.
Since the quantile $z\skl{\alpha}$ is defined as the solution  of  $\exp \kl{-\exp\kl{-z}} = 1- \alpha$ this yields
Equation~\req{confidence}.
\end{proof}

\begin{corollary}\label{cor:confidence}
Let $\skl{\dict_{n}}_{n \in \N}$ be any family of frames (not necessarily asymptotic stable) having  normalized elements,
and consider the data model $\Data_n = \Para_n +  \Base_n\noise_n$ with noise vectors $\noise_n$ having possibly dependent $N\mkl{0,\sigma^2}$-distributed entries.
Then,  it still holds that
\begin{equation}\label{eq:confidence-weak}
\liminf_{n \to \infty}
\wk \set{
\Para_n \in \ball\mkl{\Data_n, T \mkl{\alpha_n, \abs{\Om_n}}}
\,;
\forall \Para_n  \in \range \skl{\Base_n}
}
\geq 1-\alpha  \,.
\end{equation}
\end{corollary}

\begin{proof}
This follows from Theorem~\ref{thm:confidence}
and Lemma~\ref{lem:sidak}.
\end{proof}

Notice, that in Corollary~\ref{cor:confidence}
the sets $\ball\mkl{\Data_n, T \mkl{\alpha_n, \abs{\Om_n}}}$ are not necessarily asymptotically sharp confidence regions, in the sense that inequality~\req{confidence-weak} may be strict.
Actually,  we believe that  asymptotical stability of the frames $\dict_n$ is close to being  necessary for the sets $\ball\mkl{\Data_n, T \mkl{\alpha_n, \abs{\Om_n}}}$ defining asymptotically sharp confidence regions.
For specific highly  redundant  dictionaries where
asymptotic stability fails to hold
(such  as the translation invariant wavelet frame;
see Section~\ref{sec:ti})  we expect that
$ \wk \mset{ \snorm{\Base_n\noise_n}_\infty
\leq  \sigma a_nz + \sigma b_n }$ still converges to the Gumbel distribution
-- however with normalization sequences $a_n$ and $b_n$ being strictly smaller
than $\sigma a\mkl{\chi, \abs{\Om_n}}$ and $\sigma b\mkl{\chi, \abs{\Om_n}}$.
If this is the case, then the smaller  thresholds
$ T_n = \sigma a_nz \skl{\alpha_n} +  \sigma b_n $ again define sharp confidence regions.  Surprisingly,
results on the distributional convergence of
$\snorm{\Base_n\noise_n}_\infty$ or even
of $\max\mkl{\Base_n\noise_n}$
for redundant frames are almost nonexistent.

\subsection{Rate of Approximation}

Strictly taken, Theorem~\ref{thm:confidence} only claims
that the $\enorm{}_\infty$-balls $\ball\mkl{\Data_n, T \mkl{\alpha_n, \abs{\Om_n}}}$ turn into
confidence regions in the limit $n \to \infty$,
but it does not  directly give any result for finite $n$.
Sometimes it is argued that, even in the independent case without taking absolute values,
the rate of  convergence of  $\wk\mset{ \max\skl{\Base_n\noise_n} \leq T} $ to  the  Gumbel distribution is  known to be
rather  slow (see, for example, \cite[Section~2.4]{LeaLinRoo83}). Another option could be
to derive non-asymptotic coverage probabilities
along the lines of \cite{KerNicPic12}, however at the price of typically quite conservative confidence bands.

\begin{figure}[htb!]\centering
\includegraphics[width=0.3\textwidth,height=0.23\textwidth]{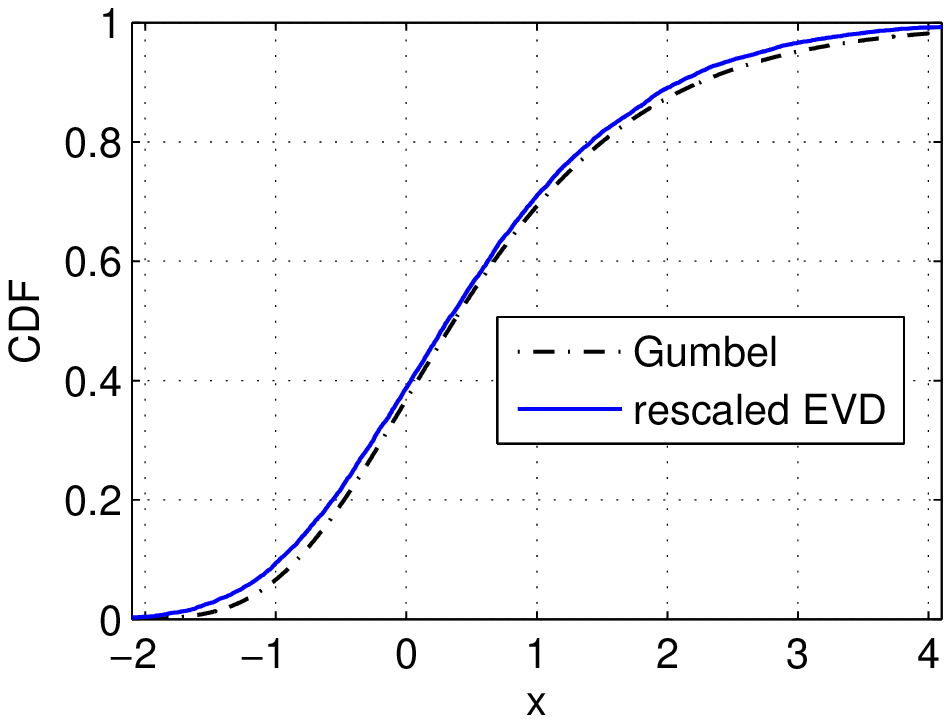}
\includegraphics[width=0.3\textwidth,height=0.23\textwidth]{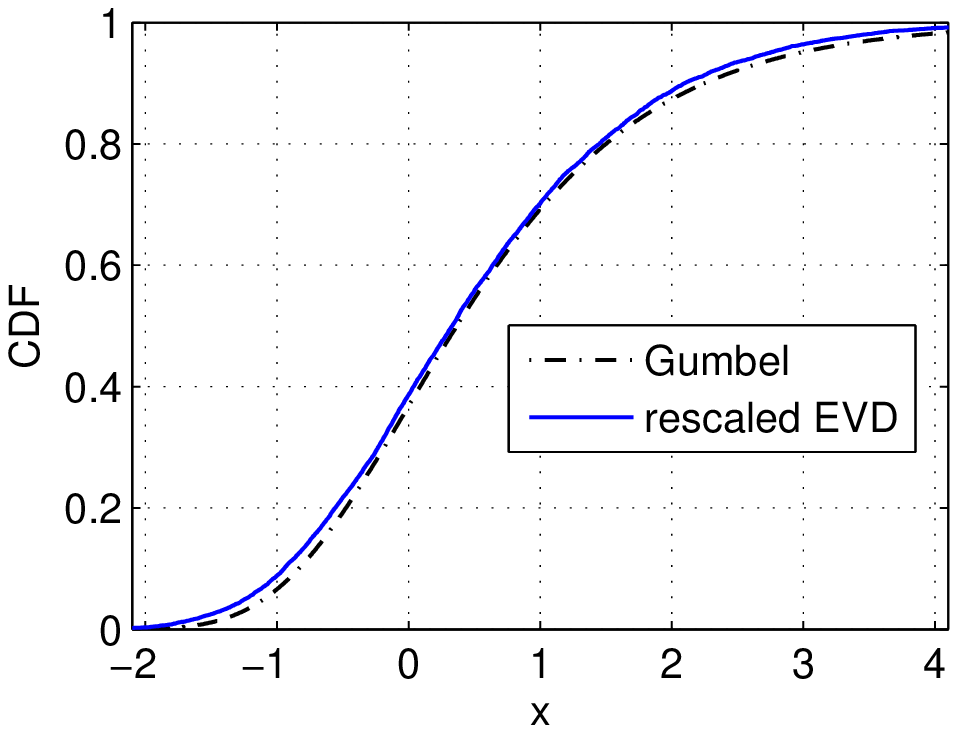}
\includegraphics[width=0.3\textwidth,height=0.23\textwidth]{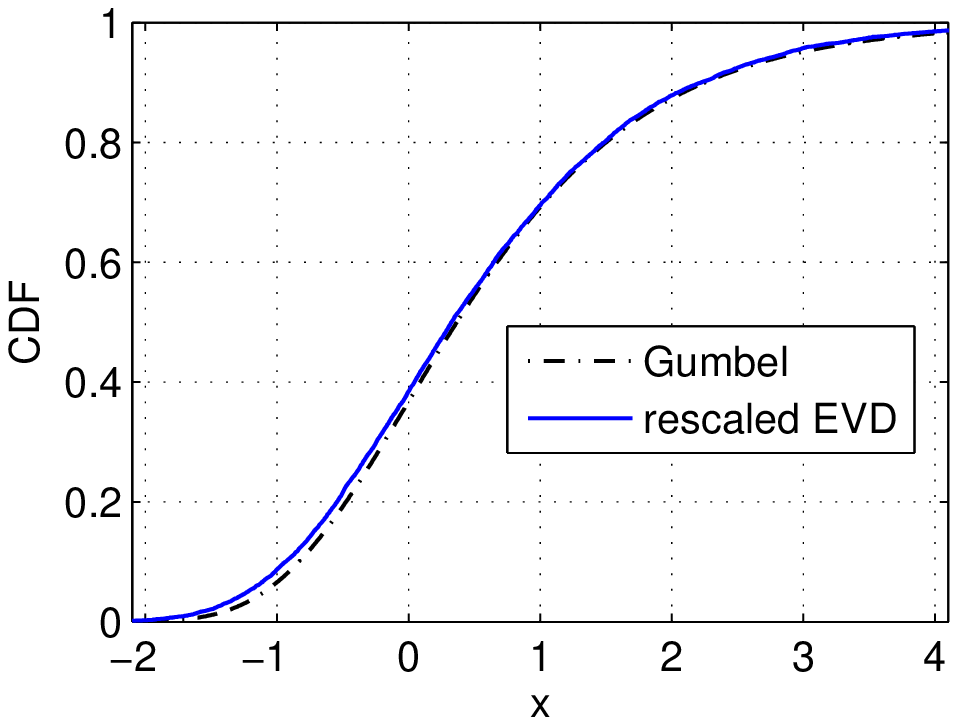}
\\
\includegraphics[width=0.3\textwidth,height=0.23\textwidth]{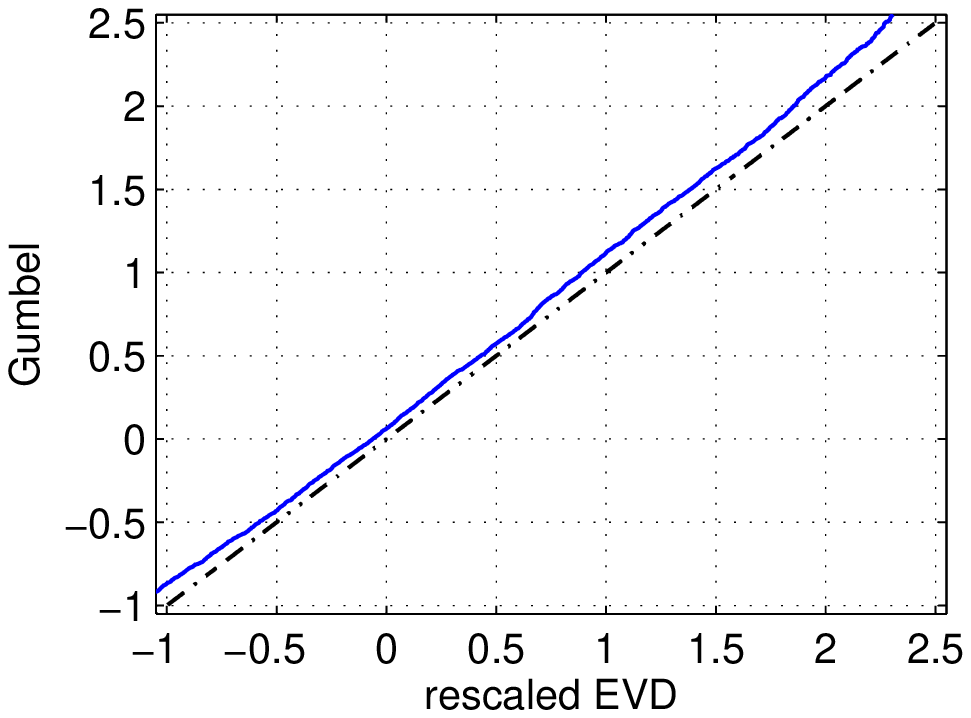}
\includegraphics[width=0.3\textwidth,height=0.23\textwidth]{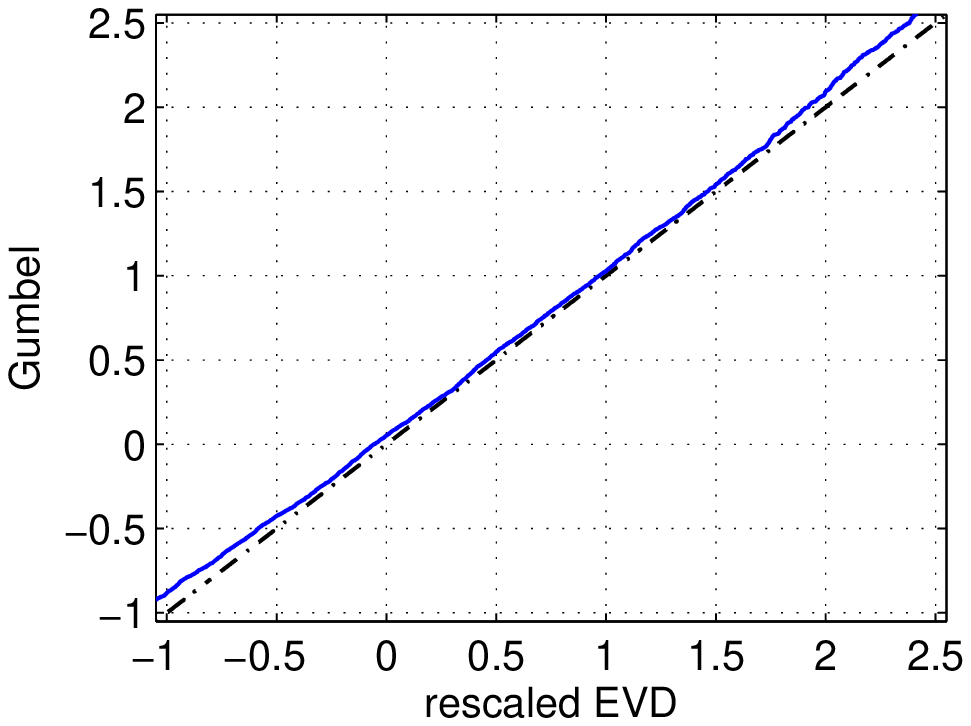}
\includegraphics[width=0.3\textwidth,height=0.23\textwidth]{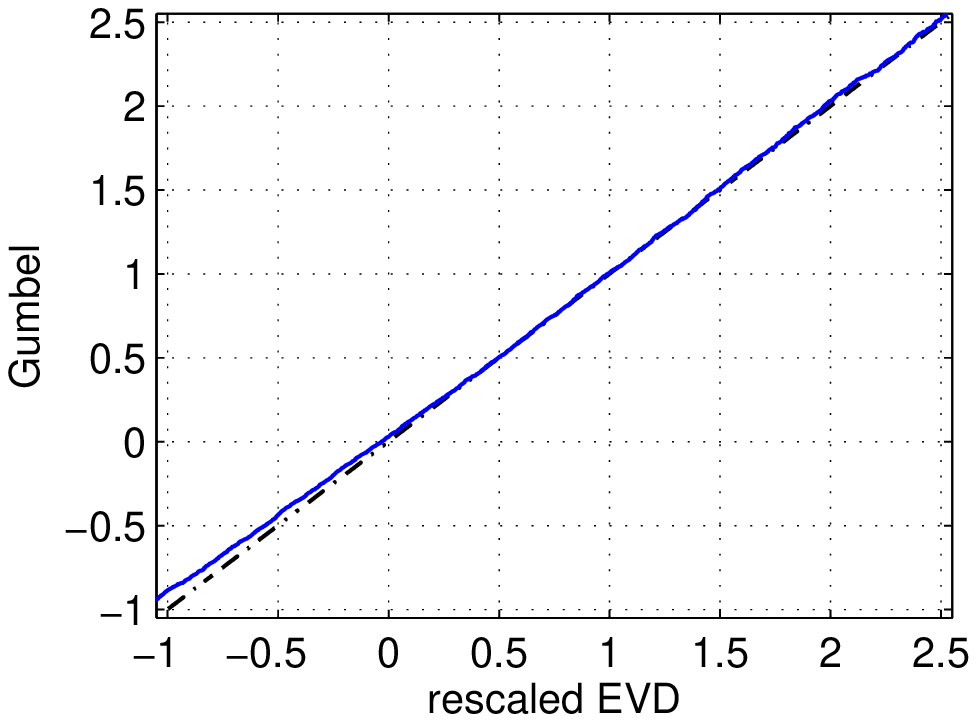}
\caption{Top: Rescaled distribution of
$\snorm{\Base_n\noise_n}_\infty $ and the Gumbel distribution for
$n=128$ (left), $n=512$ (middle) and $n=1024$ (right) .
Bottom: Q-Q-plot of those distributions.\label{fig:sine-dist}}
\end{figure}

Nevertheless, numerical simulations clearly demonstrate,
that even for moderate $n$, the approximation of
$\wk\mset{
\snorm{\Base_n\noise_n}_\infty
\leq
\sigma a\skl{\chi,\abs{\Om_n}} z  +
\sigma b\skl{\chi,\abs{\Om_n}} } $ with the limiting Gumbel distribution is quite good. This even holds true for redundant frames as  can be  seen from Figure~\ref{fig:sine-dist}, where   the distribution functions of the rescaled  maxima of the coefficients with respect to the  two times oversampled sine frame of Example~\ref{ex:sine2} are  compared with the limiting Gumbel distribution.  The top line in Figure~\ref{fig:sine-dist} displays the normalized empirical distributions of $\snorm{\Base_n\noise_n}_\infty$  for signal lengths of $n=128$, $n=512$ and $n=1024$ (computed from  $10000$ realizations in each case) and the limiting Gumbel distribution. As can be seen, there is only a small
difference between those  functions. The bottom line in  Figure~\ref{fig:sine-dist} shows a Q-Q-plot (quantile against quantile)  of those distributions and again indicates that the quantiles of  the rescaled maximum for finite $n$ are quite well approximated  by the ones of the limiting Gumbel distribution.

\subsection{Smoothness Estimates}
\label{sec:smoothness}

We have just seen that  the $\Para_n$ is contained in the
confidence regions  $\ball\skl{\Data_n, T \skl{\alpha_n, \sabs{\Om_n}}}$ around the data $\Data_n$ with probability tending to $1-\alpha$. Moreover, by definition,  the soft-thresholding estimate $\hat \Para_n = \softop \kl{ \Data_n, T \skl{\alpha_n, \sabs{\Om_n}}}$ is contained in $\ball\skl{\Data_n, T \skl{\alpha_n, \sabs{\Om_n}}}$, too.
The following theorem shows that the soft-thresholding
estimate is indeed the smoothest element in this confidence region, with smoothness measured in terms of a wide class of functionals.

\begin{theorem}[Smoothness estimates]\label{thm:smooth}
Let $\mkl{\J_n}_{n \in \N}$ be a family of functionals
$\J_n \colon \R^{\Om_n} \to \R \cup \set{\infty}$
having the property that
\begin{equation}\label{eq:Jmono}
\J_n \mkl{\Para_n} \leq \J_n \mkl{\bar\Para_n}
\; \text{ whenever } \;
\sabs{\Para_n\skl{\om}} \leq \sabs{\bar\Para_n\skl{\om}}
\; \text{ for all } \; \om \in \Om_n \,.
\end{equation}
Moreover, consider the data model $\Data_n = \Para_n +
\Base_n\noise_n$, where $\skl{\noise_n}_{n\in\N}$ is a sequence of random vectors with $N\skl{0,\sigma^2}$-distributed entries,
let  $\skl{\alpha_n}_{n \in \N}$ be a sequence
in $\skl{0,1}$  converging to some $\alpha\in [0,1)$,
and denote
$\hat\Para_n := \softop\mkl{ \Data_n, T \skl{\alpha_n, \sabs{\Om_n}} }$.
Then,
\begin{equation}\label{eq:smooth}
\liminf_{n\to\infty}
\wk \set{
\J_n\mkl{\hat\Para_n}
\leq  \J_n\mkl{\Para_n}
\,;
\forall \Para_n  \in \range \skl{\Base_n}
} \geq
1 -  \alpha \,.
\end{equation}
Hence,  the soft-thresholding estimate $\hat\Para_n$ is at least  as smooth as
the original parameter $\Para_n$, with probability tending to $1-\alpha$ as $n \to \infty$, where
smoothness is measured in terms of  any  family of
functionals  $\J_n$ satisfying~\req{Jmono}.
\end{theorem}

\begin{proof}
The definition of the soft-thresholding function implies that $\hat\Para_n$ is an element of the confidence region $\ball\mkl{\Data_n, T \mkl{\alpha_n, \abs{\Om_n}}}$
and that for every  other element $\bar\Para_n$ contained in this  confidence region we have $\sabs{\hat\Para_n\skl{\om}} \leq \sabs{\bar\Para_n \skl{\om}}$ for all $\om \in \Om_n$.
By Corollary~\ref{cor:confidence} the true parameter
$\Para_n$ is contained in  $\ball\mkl{\Data_n, T \mkl{\alpha_n, \abs{\Om_n}}}$, too, with a probability tending to $1-\alpha$.
We conclude that
\begin{equation}\label{eq:shrinkage}
\liminf_{n \to \infty}
\wk\set{
\sabs{\hat\Para_n\skl{\om}} \leq \sabs{\Para_n\skl{\om}}
\,;
\om \in \Om_n
\,;
\forall \Para_n  \in \range \skl{\Base_n}
}
\geq 1- \alpha  \,.
\end{equation}
The  assumption~\req{Jmono} on component-wise monotonicity
of  the functionals  $\J_n$  now implies that the event
$\mset{\sabs{\hat\Para_n\skl{\om}} \leq \sabs{\Para_n\skl{\om}} \,;
\forall \om \in \Om_n
\,;
\forall \Para_n  \in \range\skl{\Base_n}   }$
is contained in the event $\sset{ \J_n\skl{\hat\Para_n} \leq  \J_n\skl{\Para_n} \,; \forall \Para_n  \in \range\skl{\Base_n}}
$. Together with~\req{shrinkage} this  yields~\req{smooth}.
\end{proof}

\begin{remark}[Shrinkage property]
The proof of Theorem~\ref{thm:smooth} uses two main ingredients: First,  soft-thresholding selects  that element in $\ball\mkl{\Data_n, T \skl{\alpha_n, \sabs{\Om_n}}}$ which has minimal component-wise magnitudes and second,
the true coefficient $\hat\Para_n$ is contained in the set $\ball \mkl{\Data_n, T \skl{\alpha_n, \sabs{\Om_n}}}$ with probability tending to $1-\alpha$. The former property is often referred to as the   \emph{shrinkage property} of soft-thresholding and has already been  used in~\cite{Don95} for deriving  smoothness estimates for orthogonal wavelet soft-thresholding.
The second property relies on our extreme value
result derived in Theorem~\ref{thm:main}.
Notice,  that  the weaker result
$\wk \sset{ \J_n\skl{\hat\Para_n} \leq \J_n\skl{\Para_n} } \to 0$ using  the
threshold $\sigma\sqrt{2\log \abs{\Om_n}}$ is  well known;  compare  \cite{JohSil97}.
However, the proof of Theorem~\ref{thm:smooth} reveals that for asymptotically stable
frames the considered thresholds $T \skl{\alpha_n, \sabs{\Om_n}}$
are close to being the smallest ones yielding smoothness estimates of the form \req{smooth}.
For strongly redundant frames,  however, where asymptotic stability fails to hold, smaller thresholds  yielding the same smoothness bounds can exist.  In Theorem~\ref{thm:ti} we show that this is indeed the case for the dyadic discrete translation invariant wavelet system.
\end{remark}

Basic but important examples for functionals satisfying
the component-wise monotonicity property~\req{Jmono} are powers of weighted   $\ell^2$-norms,
\begin{equation*}
\mnorm{\Para_n}_2
:=
\sqrt{\sum_{\om \in \Om_n}
c\skl{\om} \, \mabs{\Para_n\skl{\om} }^2}
\qquad \text{ for some  } \;  c\skl{\edot} >0 \,.
\end{equation*}
In the case of wavelet and  Fourier frames, these norms
of the coefficients provide norm equivalents to Sobolev norms in the original signal domain (assuming an appropriate  discretization model  $\signal \mapsto \signal_n$).
Sobolev norms are definitely the most basic smoothness measures of functions. More general and also practically relevant  classes of smoothness measures  are Besov norms.  Assume for the moment  that $\dict_n$ is a wavelet frame  where  the index set has the  multiresolution form
$\Om_n = \mset{\kl{\la,k}: \la \in \La_n \text{ and } k\in D_\la}$ for some index sets $\La_n$ and $D_\la$ corresponding to scale/resolution and scale dependent location, respectively.
In this case one takes   the functional $\J_n$
as one of the weighted $\ell^{p,q}$-norms
\begin{equation*}
\mnorm{\Para_n}_{p,q}
:=
\sqrt[q]{\sum_{\la \in \La_n}
c\kl{\la} \mnorm{\Para_n\kl{\la, \cdot} }^q_p}
\qquad \text{ for  some } \;  c\kl{\edot} >0 \,.
\end{equation*}
These norms again satisfy the monotonicity
property~\req{Jmono} and moreover yield to
norm equivalents  of  Besov norms for properly  chosen weights  $ c\kl{\la}$; see Section~\ref{sec:wave}.
Such weighted $\kl{p,q}$-norms are also reasonable  in
combination with other multiresolution systems,
such as the curvelet frame (see Section~\ref{sec:curv}).

\subsection{Risk Estimates}

Although the  main focus in this work is on confidence regions and  smoothness estimates, in the following
Proposition~\ref{prop:risk} we shall verify that
using the EVTs of Definition~\ref{def:alphathresh} yields  risk estimates  similar to the oracle inequalities of~\cite{DonJoh94}.
The following result is non-standard regarding  two aspects:
First, it allows arbitrary frames instead of
orthonormal bases. Second, and  more importantly,
it considers our more general class of extreme value
thresholds instead of the universal threshold $\sigma \sqrt{2 \log\abs{\Om_n}}$.

\begin{proposition}[Oracle inequality]\label{prop:risk}
Let $\dict_n = \kl{\base_\om^n\colon \om \in \Om_n}$ be a frame in $\R^{\I_n}$ with
corresponding analysis operator
$\Base_n$.
Moreover, let  $\hat\signal_n =
\Base_n^+ \circ \softop \kl{\Base_n \skl{\data_n}, T} $ denote  the soft-thresholding
estimator in \req{est-soft} corresponding to the extreme value thresholds
$T  =  T \mkl{\alpha_n, \abs{\Om_n}}$ defined by Equation~\req{alphathresh}, and assume for simplicity that
$T \mkl{\alpha_n, \abs{\Om_n}} \leq  \sigma \sqrt{2 \log\abs{\Om_n}}$.
Then, we have
 \begin{multline} \label{eq:oracle}
\ew \kl{\norm{\signal_n -\hat\signal_n}^2 }
\leq
\frac{\sigma^2}{a_n} \, \biggl(  \log \kl{1/\skl{1-\alpha_n}}
\sqrt{\pi \log \abs{\Om_n} }
\\
+
\kl{1+ 2\log \abs{\Om_n}}
\sum_{\om \in \Omega_n}
\min\set{1,  \frac{\sabs{ \sinner{\base_{\om}^{n}}{\signal_n}}^2}{\sigma^2} }
\biggr)
\,.
\end{multline}
Here  $a_n$ is the lower frame bound of $\dict_N$;
see Equation  \req{frame-prop}.
\end{proposition}

\begin{proof}
Section~\ref{ap:risk}.
\end{proof}

\section{Examples from Signal and Image Denoising}
\label{sec:ez}

In this section we verify that many important
frames used for thresholding in signal and image  processing  are asymptotically stable and thus covered by the results of the previous section.
These examples include redundant wavelet systems
and curvelet frames. We also consider an important example, where our basic asymptotic stability fails to hold; namely  the discrete translation invariant wavelet frame.  Actually, we show  that  not even the result of Theorem~\ref{thm:main} (and thus all of  its implications) holds in this case. This  indicates that
the stated conditions are close to being necessary for
the asymptotical distributional law of Theorem~\ref{thm:main}.
Further, we derive confidence regions and smoothness estimates for the translation invariant wavelet transform
that significantly  improve over simple application of
Proposition~\ref{prop:denoise},  Item~\ref{it:denoise2}
(and also the main result of~\cite{BerWel02}).

\subsection{Redundant and Non-Redundant
Wavelet Denoising}
\label{sec:wave}

In the following we  consider  one  dimensional wavelet
denoising. The generalization to higher dimensional wavelet
denoising is straightforward.
We shall discuss thresholding in biorthogonal wavelet bases,
certain overcomplete wavelet frames (using the so called cycle spinning procedure),
and fully translation invariant wavelet systems.
Before considering those particular examples,
we collect some notation and present basic
facts about  biorthogonal wavelets (which include the orthogonal ones) that we need for the application of our general results.

\subsubsection{Biorthogonal Wavelet Bases}
\label{sec:wave-setting}

One dimensional wavelets are generated by  dilating and translating a single   function, the so called mother wavelet. The  distinguished feature of  wavelet systems is that various classical smoothness measures (Triebel, Sobolev  and Besov norms) can be characterized by simple norms in the wavelet domain.  In the following, for the sake of simplicity, we  only consider real valued periodic wavelets on the interval $[0,1]$. Moreover, we restrict ourselves to  compactly supported  biorthogonal wavelets that arise from a multiresolution decomposition.

Denote by  $\Om$ the set of all index pairs  of the form
$\skl{j,k}$ with $j \in \N$ and $k \in \sset{0, \dots, 2^{j}-1}$.  The index $j$ is refereed to as resolution or scale index and $k$ to as the discrete location index. Moreover, let $\scale, \wave \in L^2\kl{\R}$  denote the father and mother wavelet,
 respectively, which are assumed to be compactly supported and to have unit norm with respect to $\enorm{}_2$, the Euclidian norm on $L^2\kl{0,1}$.  For any $\skl{j,k} \in \Om$ one then defines (periodic) wavelets  $\wave_{j,k}$ and (periodic) scaling functions
 $\scale_{j,k}$  on $[0,1] $by
\begin{equation*}
\kl{\forall t \in [0,1]}
\qquad
\wave_{j,k}\skl{t}     = 2^{j/2}
\sum_{m\in \Z}\wave \kl{2^{j} \kl{t-m} - k}
\,, \quad
\scale_{j,k}\skl{t}     = 2^{j/2}
\sum_{m\in \Z}\scale \kl{2^{j} \kl{t-m} - k} \,.
\end{equation*}
The wavelet and the scaling coefficients of  some signal
$u\in L^2\kl{0,1}$ are then simply the
inner products  of $u$ with the wavelets $\wave_{j,k}$ and the scaling functions $\scale_{j,k}$, respectively.
We further write $\Wave, \Scale  \colon  L^2\kl{0,1}  \to \ell^2\kl{\Om}$ for the mappings that take  the signal
$\signal \in L^2\kl{0,1}$ to the inner products
$\kl{\Wave\signal} \skl{j,k} := \minner{\wave_{j,k}}{\signal}$  and $\kl{\Scale\signal} \skl{j,k} := \minner{\scale_{j,k}}{\signal}$, respectively.

In order to get a (biorthogonal) wavelet basis
one has to impose some completeness condition and
some connections between  the wavelets and the scaling functions.  Such assumptions are most naturally formulated in the multiresolution framework (below already adapted to the periodic setting). Hence, in the following we assume the  existence of subspaces $ \V_{j}$ and $ \W_{j}$ of $L^2\kl{0,1}$, referred to as scaling and detail spaces, respectively, meeting the following requirements:
\begin{itemize}[topsep=0em]
\item
For every $j\in\N$, the  following
mappings are bijections:
\begin{align*}
\V_{j}  \to \R^{2^j}
&\colon
\signal \mapsto \kl{\minner{\scale_{j,k}}{\signal}
: k \in \sset{0, \dots, 2^{j}-1}} \,,
\\
\W_{j}  \to \R^{2^j}
&\colon
\signal \mapsto \kl{\minner{\wave_{j,k}}{\signal}:
k \in \sset{0, \dots, 2^{j}-1} } \,.
\end{align*}

\item
For every $j \in \N$,
we have the multiresolution decomposition
$\V_{j} = \V_{j-1} \oplus \W_{j-1}$.

\item
The union $\bigcup_{j\in \N}\V_{j}$ is dense in
$L^2\kl{0,1}$.
\end{itemize}

Repeated application of the multiresolution decomposition
yields the decomposition of the signal space into the  sum of the scaling
space  $\V_0$ and the wavelet space $\W := \bigoplus_{j\geq0} \W_j$.
Moreover, the above conditions imply that there is
a stable one to one correspondence between  any element in $\W$
and its inner product with respect to
$\dict := \mkl{ \wave_{j,k}: \skl{j,k} \in \Om }$.
Moreover,  the multiresolution decomposition serves as
the basis of  both, discretization and fast implementation.
Notice that the construction  of compactly supported orthogonal and
biorthogonal wavelets is non-trivial
and  such systems have been constructed for the
first time in~\cite{CohDauFea92,Dau88}.
By now such wavelet systems are well known; a
detailed construction of orthogonal and
biorthogonal wavelet systems together with  many interesting details may be found in~\cite{Coh03,Dau92,Mal09}.

\begin{remark}[Biorthogonal  basis]\label{rem:bi}
If the spaces $\V_{j}$ and $ \W_{j}$ are orthogonal to each other, then $\dict$ is an orthonormal wavelet basis and  $ \V_{j}$ and $ \W_{j}$ are spanned by the scaling and wavelet functions, respectively. However, we do not require orthogonality in the following. In this more general case, the scaling and wavelet spaces  are spanned by certain dual systems (or biorthogonal bases; thus the name). Biorthogonal wavelets are often preferred to strictly orthogonal ones since they allow more freedom to adapt them to a particular application in mind. Especially, opposed to orthogonal wavelets,  biorthogonal wavelets  can at the same time be smooth, symmetric and compactly supported.
\end{remark}

\begin{remark}[Computing the wavelet transform]\label{rem:mr}
The  multiresolution decomposition  $\V_{j} = \V_{j-1} \oplus  \W_{j-1}$ is the basis for fast computation of the wavelet transform. Given the scaling coefficients at some scale  $L>0$, the scaling and wavelet coefficients  at scale  $L-1$ can be computed by cyclic convolution of the given scaling coefficients  with a certain discrete filter pair. Repeated application of this procedure eventually yields all scaling and all wavelet coefficients at scales below  $L$. In the case of biorthogonal  wavelets, the   multiresolution decomposition can  be inverted again by repeated  application of convolution with a different  pair of reconstruction filters.
\end{remark}

Throughout the  following  we assume  that a discrete signal
$\signal_n \in \R^{n}$ is given, where the discretization number $n =2^J$ is an integer power of some maximal level of resolution.
One  then  interprets the components  of the discrete signal
as the scaling coefficients of some underlying continuous domain signal, that is,
\begin{equation*}
\kl{\forall k\in \set{0, \dots,  n - 1}}
\qquad
	\signal_n \kl{k}
	=
	\kl{\Scale\signal}
	\kl{J,k}
	=
	\inner{\scale_{J,k}}{\signal} \,.
\end{equation*}
Obviously there are infinitely many continuous domain
signals  yielding to the same scaling coefficients.
However, according to the made assumptions, there exists  a unique element
in the scaling space $\V_J$ having scaling  coefficients  $\signal_n$.
This element will be denoted  as $\signal_n^*\in \V_J$.

The wavelet coefficients of the discrete signal
are then simply defined as the wavelet
coefficients of  the continuous domain signal  with indices
in
\begin{equation*}
\Om_n := \set{\skl{j,k}:
j \in \sset{ 0, \dots, J-1} \text{ and }
k \in \sset{0, \dots, 2^j-1 }}   \,.
\end{equation*}
According to the multiresolution decomposition, these coefficients depend only on  the discrete signal and
can also be written as discrete inner products
\begin{equation*}
\skl{\forall u_n\in \R^{n}}
\skl{\forall \skl{j,k} \in \Om_n}
\qquad
	\inner{\wave_{j,k}^n}{\signal_n}
	:=
	\inner{\wave_{j,k}}{\signal} \,.
\end{equation*}
This serves  as definition of both, the discrete  wavelets
$\wave_{j,k}^n \in \R^{n}$ and the wavelet coefficients of $\signal_n$.
Finally we define $\dict_n$ as the family of all discrete wavelets $\wave_{j,k}^n$ and denote by
$\Wave_n\colon \R^{n} \to \R^{\Om_n}$ the
corresponding analysis operator, which we refer to as  the
 \emph{discrete wavelet transform}.

\begin{remark}[Numerical computation]
The discrete wavelet transform is computed by repeated
application of the multiresolution decomposition, yielding to all discrete wavelet coefficients and the scaling coefficient $\signal_1 =\inner{\scale_{0,0}}{\signal}$; see Remark~\ref{rem:mr}.
Since  the discrete filters usually have small support,
the  wavelet transform can be computed using only
$\mathcal O\kl{n}$ operation counts and the
 same holds true for recovering $\signal_n$ from those coefficients. Notice that  the discrete  wavelets are never computed explicitly in the multiresolution algorithm.
We defined  them in order to verify our general framework.
Finally, we stress again that the   wavelet  coefficients
of $\signal_n$  coincide with the one of $\signal$ up to scale
$\log\kl{n/2}$.
\end{remark}

\subsubsection{Biorthogonal Basis Denoising}

Now consider the denoising problem~\req{problem},
which simple reads $\data_n =\signal_n +\noise_n$.
The wavelet soft-thresholding procedure is usually only applied to  coefficients above some scale; compare with Remark~\ref{rem:frame}.
For  simplicity we shall consider the case where all  wavelet coefficients are thresholded but not the scaling coefficient. Hence, the wavelet  soft-thresholding estimator (for the wavelet part of $u_n$) is  defined by
\begin{equation*}
\hat u_n
=
\Wave_n^{-1} \circ \softop \kl{\Wave_n \data_n, T}  \,.
\end{equation*}
Thanks to the multiresolution  algorithm,
the  wavelet soft-thresholding estimator  can be
computed with only $\mathcal O\kl{n}$  operation
counts.

We measure smoothness of the  considered estimates in
terms of  Besov norms. To that end, assume that the mother
wavelet has sufficiently  many vanishing moments and is sufficiently smooth. Then,  for given norm parameters $p$, $q \geq 1$ and given smoothness parameter
$r \geq 0 $
one  defines
\begin{equation*}
\mnorm{\Para}_{p,q,r}
=
\sqrt[p]{\sum_{j\in\N}
2^{jsq} \,
\mnorm{ \Para\skl{j,\edot}}_p^q}
\qquad
\text{ with } \quad
s = r +  \frac{1}{2} - \frac{1}{p}  \,.
\end{equation*}
for any $\Para \in \ell^2\kl{\Om}$ and with
$\enorm{}_p$  denoting the usual
$\ell^p$-norm taken for fixed scale $j \in\N$.
It is then clear that  any of these norms
satisfies the component-wise monotony
property~\req{Jmono}.  We further write
$\norm{\signal}_{\B_{p,q}^r} :=
\norm{\Wave\signal}_{p,q,r}$ for the corresponding norm
of some
$\signal \in L^2\kl{0,1}$ and finally denote by
$\B_{p,q}^r$ the set of all signals having finite norm
$\norm{\signal}_{\B_{p,q}^r} <  \infty$.
The pair $\mkl{\B_{p,q}^r, \enorm{}_{\B_{p,q}^r}}$
is a Banach space and referred to as Besov space.  The given
definitions provide  norm equivalents of
$\enorm{}_{\B_{p,q}^r}$ to the definition of Besov norms in
classical  analysis,
as long as the mother wavelet  has  $m > r$
vanishing moments and is $m$-times continuously
differentiable.

\begin{theorem}[Soft-thresholding in wavelet bases]\label{thm:wave}
The discrete wavelet bases $\dict_n =
\mkl{\wave_{j,k}^n: \skl{j,k} \in \Om_n}$
are asymptotically stable.
In particular, the following holds:
\begin{enumerate}[label=(\alph*),topsep=0.0em]
\item \label{it:wave1}
\textbf{Distribution:}
Let $\noise_n$ be a sequence of noise vectors  in
$\R^{n}$ with independent
$N\skl{0,\sigma^2}$-distributed entries.
Then, the sequence $\snorm{\Wave_n\noise_n}_\infty$
is of Gumbel type with normalization constants
$\sigma a\kl{\chi,n}$,
$\sigma b\kl{\chi,n}$ defined by \req{a-abs}, \req{b-abs}.

\item \label{it:wave2}
\textbf{Confidence regions:}
Let $\skl{\alpha_n}_{n \in \N} \subset \kl{0,1}$
be a sequence  of significance levels converging to some
$\alpha \in [0,1)$ and let
$T\mkl{\alpha_n, n}$ denote the
corresponding  EVTs defined in~\req{alphathresh}.
Then,
\begin{equation*}
\lim_{n \to \infty} \wk \set{
 \snorm{
\Wave_n\skl{\signal_n - \data_n}}_\infty \leq T \mkl{\alpha_n, n}
\,;
\forall \signal_n   \in \R^{\I_n}
}  =  1 -  \alpha \,,
\end{equation*}

\item \label{it:wave3}
\textbf{Smoothness:}
Let $\hat \signal_n^*$ denote the soft-thresholding estimator using the extreme value thresholds $T \mkl{\alpha_n, n}$. If the considered mother wavelet
has $m>r$ vanishing moments and is $m$ times continuously differentiable, then
\begin{equation*}
\liminf_{n\to\infty}
\wk\set{
\snorm{\hat\signal_n^*}_{\B_{p,q}^r}
\leq  \snorm{\signal}_{\B_{p,q}^r}
\,;
\forall \signal  \in \B_{p,q}^r
} \geq   1 -  \alpha \,.
\end{equation*}
\end{enumerate}
\end{theorem}

\begin{proof}
By definition, for any $n \in \N$ and pair of any indices
$\skl{j,k}, \skl{j',k'}$, the inner products $\sinner{\wave_{j,k}^n}{\wave_{j',k'}^n}$ of the discrete wavelets coincide with the inner product
$\sinner{\wave_{j,k}}{\wave_{j',k'}}$ of the continuous domain wavelets.
Since the family $ \mkl{\wave_{j,k}:  \skl{j,k} \in \Om }$ is a  Riesz  basis
with normalized elements this  immediately yields Condition~\ref{it:frame3}
required in the Definition~\ref{def:stab} for  asymptotically stable frames.

Condition~\ref{it:frame2} of  Definition~\ref{def:stab} is satisfied since
all $\sabs{\sinner{\wave_{j,k}}{\wave_{j',k'}}}$
are bounded away from  one.  To see that this holds true,  it is sufficient to
consider the case
where  $\wave \skl{2^jt-k}$ and $\wave \skl{2^{j'}t-k'}$ are both supported in the interval $\skl{0,1}$ and satisfy $j'\leq j$. Application of the substitution rule yields
\begin{multline*}
\sabs{\sinner{\wave_{j,k}}{\wave_{j',k'}}}
=
2^{j/2+j'/2}
\abs{\int_{\R}
\wave \skl{2^j t - k} \,
\wave \skl{2^{j'}t-k'}  \, dt}
\\
=
2^{\skl{j-j'}/2}
\abs{\int_{\R}
\wave \skl{2^{j-j'} t - k + 2^{j-j'} k} \,
\wave \skl{t}  \, dt }
=
\sabs{\sinner{\wave_{j-j',k-2^{j-j'} k}}{\wave }}
\,.
\end{multline*}
Because all wavelets have unit norm, the  Cauchy-Schwarz inequality shows
$\sabs{\sinner{\wave_{j-j',k-2^{j-j'} k}}{\wave}} < 1$.   The upper frame bound  implies  that $\sum_{\skl{j,k} \in \Om_n}\sabs{\sinner{\wave_{j,k}}{\wave}}^2 <  \infty $, and hence  the sequence $\mkl{\sinner{\wave_{j,k}}{\wave}: \kl{j,k} \in \Om_{n}}$   in particular converges to zero. As a consequence, the numbers $\sabs{\sinner{\wave_{j,k}}{\wave_{j',k'}}}$
are uniformly bounded by some constant $\rho < 1$.

The other claims in Items~\ref{it:wave1}--\ref{it:wave3} then follow from the asymptotic stability of  the frames $\dict_n$ and the general results derived the previous section.
Actually, the first two items are just restatements of
Theorems~\ref{thm:main} and~\ref{thm:confidence}
adapted to the wavelet setting.
For Item~\ref{it:wave3} note that the norms $\enorm{}_{p,q,r}$ satisfy the component-wise monotony property~\req{Jmono} and therefore   Theorem~\ref{thm:smooth}
yields
\begin{equation*}
\liminf_{n\to\infty}
\wk\set{
\snorm{\hat\Para_n}_{p,q,r} \leq
\snorm{\Wave_n \signal_n}_{p,q,r}
\,;
\forall
\signal  \in \B_{p,q}^r
} \geq
1 - \alpha
\; \text{ with }  \;
\hat\Para_n
:= \softop \kl{ \data_n, T \mkl{\alpha_n, \abs{\Om_n}} }\,.
\end{equation*}
By definition  we have
$
\snorm{\hat \Para_n}_{p,q,r}
=
\snorm{\hat \signal_n^*}_{\B_{p,q}^r}
$
and the inequality
$\snorm{\Wave_n \signal_n}_{p,q,r}
\leq \snorm{\signal}_{\B_{p,q}^r}$ for all $n$
which finally yields Item~\ref{it:wave3} and concludes the proof.
\end{proof}

\subsubsection{Cycle Spinning}
\label{sec:cs}

A mayor drawback of   thresholding in a wavelet basis is its
missing translation invariance.
This typically causes visually disturbing Gibbs-like artifacts near discontinuities at non-dyadic locations.
One way to significantly reduce  these  artifacts
is via so called cycle spinning (see~\cite{CoiDon95}).
The idea there is to reduce the artifacts  by averaging
several estimates obtained by denoising shifted copies
of the noisy data.

Let  $ \dict_n = \mkl{\wave_{j,k}^n: \skl{j,k} \in \Om_n}$ be an orthonormal wavelet basis and   denote by $\shiftop_m \colon \R^{n} \to \R^{n}$  the cyclic translation operator, defined by  $\skl{\shiftop_m \signal_n}\skl{k} = \signal_n\skl{k-m}$ for
$\signal_n\in \R^{n}$ and all $k,m \in  \set{0,\dots,  n - 1}$.
Cycle spinning  then  averages  the  wavelet  soft-thresholding  estimates  of the translated data $\shiftop_m \signal_n$ over all
shifts $m = 0, \dots, M-1$, where $M$ is some prescribed number of
considered translations. Hence, one defines
\begin{equation}\label{eq:est-cs}
\hat \signal_{n,M}
:=
\frac{1}{M}
\sum_{m=0}^{M-1}
\shiftop_{-m}  \Wave_n^* \circ
\softop \kl{ \Wave_n \shiftop_m \data_n , T}
 \,.
\end{equation}
The following elementary Lemma~\ref{lem:cs} states that
the cycle spinning estimator~\req{est-cs}
is equal to  the  soft-thresholding  estimator defined by
Equation~\req{est-soft} corresponding to the  overcomplete wavelet
frame
\begin{equation}\label{eq:dict-cs}
\dict_{n,M}
 :=
\kl{  \shiftop_{-m} \wave_{j,k} : \skl{j,k,m} \in \Om_{n,M} }
\quad \text{ with } \;
\Om_{n,M}
:=
\Om_{n} \times  \set{0, \dots, M-1 } \,.
\end{equation}
Hence wavelet cycle  spinning fits into the general
framework of soft-thresholding introduced in
Section~\ref{sec:thresh}.

\begin{lemma}\label{lem:cs}
Let $\dict_{n,M}$ be the overcomplete wavelet frame defined in~\req{dict-cs} and denote by
$\Wave_{n,M} \colon   \R^{n}  \to  \R^{n M}$ the corresponding analysis operator.
Then, the cycle  spinning  estimator~\req{est-cs}
has the representations
\begin{equation}\label{eq:cs2}
\hat \signal_{n,M}
= \frac{1}{M} \,
\Wave_{n,M}^*  \circ \softop
\kl{\Wave_{n,M} \data_n, T }
=
\Wave_{n,M}^+  \circ \softop
\kl{\Wave_{n,M} \data_n, T }
 \,.
\end{equation}
Hence the cycle spinning  estimator equals the soft-thresholding estimator corresponding to the redundant wavelet
frame $\dict_{n,M}$.
\end{lemma}

\begin{proof}
The first identity in~\req{cs2} immediately follows from \req{dict-cs} and ~\req{est-cs}.
Next we verify  the second equality.
Since the decimated wavelet transform $\Wave_n$ and the translation operators $\shiftop_m$ are isometries,
we have
\begin{equation*}
\mnorm{\Wave_{n,M} \signal_n}^2
=
\sum_{m=0}^{M-1}
\mnorm{\Wave_n \shiftop_m\signal_n}^2
= M \mnorm{\signal_n}^2
\end{equation*}
whenever  $\signal_n$ are the scaling coefficients  of a member
$\signal$  of the wavelet space  $\W$.
Hence  $\dict_{n,M}$ is a tight frame with
frame bound equals $M$.  This implies that the  dual synthesis operator $\Wave_{n,M}^+$ corresponding to the cycle spinning frame is simply given
by $\Wave_{n,M}^+ =\Wave_{n,M}^*/M$, which
 yields the second equality in~\req{cs2}.
 \end{proof}

In the following we will show that redundant cycle  spinning
frame is asymptotic stable  and thus allows the application of our
general results. Strictly taken, these conditions do not
hold for the frame $\dict_{n,M}$, since some of the elements  $\shiftop_m \wave_{j,k}$ occur more than once  in  $\dict_{n,M}$.
In particular, the cardinality of the set
$\sset{\shiftop_m \wave_{j,k}: \kl{j,k,m}
\in \Om_{n,M}}$
is strictly less than $\sabs{\Om_{n,M}} = n M$; the
 exact number of different  frame elements is
computed in the following Lemma~\ref{lem:cs-numel}.
Asymptotic stability will then be satisfied for the frame
that contains  every element $\shiftop_m \wave_{j,k}$
exactly once.

\begin{lemma}\label{lem:cs-numel}
For any $M\leq n$, the number of different elements
of the frame $\dict_{n,M}$ defined
in \req{dict-cs} is given by
\begin{equation}\label{eq:cs-numel}
\abs{\mset{\shiftop_m \wave_{j,k}: \skl{j,k,m}\in \Om_{n,M}}}
=
n  \fkl{\log_2 M}
+
 M \kl{2^{\ckl{\log_2 n/M}}  -  1}
\end{equation}
\end{lemma}

\begin{proof}
The definition of the wavelet basis implies
that $\wave_{j,k} = T_{n 2^{-j} k } \wave_{j,0}$ for very
$\skl{j,k}\in \Om_n$ and hence the periodicity of $\wave_{j,0}$ implies that
\begin{equation*}
\shiftop_m  \wave_{j,k}
=
\shiftop_{m+n 2^{-j} k}
\wave_{j,0}
=
\wave_{ j, k + m2^{j}/n}\,.
\end{equation*}
 whenever  $m 2^{j}/n$ is  an
 integer  number.
This shows that for   every given  scale index
 $j \in\set{0, \dots, \log_2  n - 1}$,
 there are  exactly $ \min\sset{n, M 2^{j}}$
 different wavelets.
 One concludes that
 \begin{align*}
 \abs{\sset{\shiftop_m \wave_{j,k}: \skl{j,k,m}\in \Om_{n,M}}}
&=
n \, \abs{\sset{j :  n/M \leq 2^{j}  \leq n/2} }
+ M \sum_{M 2^{j} < n}  2^{j}
\\
&=
n \kl{\log_2n - \ckl{\log_2 \frac{n}{M}} }
+
M
\sum_{j=0}^{\ckl{\log_2 n/M}-1 }
2^{j} \,.
\end{align*}
This shows Equation~\req{cs-numel}.
\end{proof}

In  the following we shall for simplicity assume   that $M$, the number of shifts  in the cycle spin procedure, is an  integer  power of two. In this  case, Equation~\req{cs-numel}
simplifies to
\begin{equation}\label{eq:numel-cs}
 \mabs{\sset{\shiftop_m \wave_{j,k}: \skl{j,k,m}\in \Om_{n,M}}}
= n \log_2 \kl{M} +n - M  \,.
\end{equation}
Note that this  is significantly smaller
(at least for large $M$)
than the naive bound $ M n$ given by the
cardinality of  $\Om_{n,M}$.

\begin{theorem}[Soft-thresholding using  cycle spinning] \label{thm:cs}
Let $M$ be any fixed integer power of two,  denote by
$\dict_{n,M} = \mkl{\shiftop_m \wave_{j,k}: \skl{j,k,m} \in \Om_{n,M}}$ the overcomplete wavelet cycle spinning frame and  by $\Wave_{n,M}$ the  corresponding analysis operator.
Then the following  assertions hold true:
\begin{enumerate}[label=(\alph*),topsep=0.0em]
\item \textbf{Distribution:} \label{it:cs1}
Let $\skl{\noise_n}_{n\in \N}$ be a sequence of  noise vectors  in $\R^{n}$ with independent
$N\mkl{0,\sigma^2}$-distributed entries.
Then, the sequence $\snorm{\Wave_{n,M}\noise_n}_\infty$
is of Gumbel type with normalization constants
$\sigma a\kl{\chi,n}$ (defined by~\req{a-abs}) and
$\sigma b_M\kl{\chi,n}$,  where
\begin{equation*}
	b_M\kl{\chi,n}
	 :=
	\sqrt{2 \log n}
	+
	\frac{  -  \log\log n - \log\pi + 2 \log \log_2 \kl{M}}
	{2\sqrt{2\log n}} \,.
\end{equation*}

\item \label{it:cs2}
\textbf{Confidence regions:}
Let $\alpha_n \in  \skl{0,1}$ be a sequence  of significance levels converging to some
$\alpha \in [0,1)$ and let
$T \mkl{ \alpha_n, \sabs{\dict_{n,M}}}$ denote the
corresponding  EVTs defined in \req{alphathresh}.
Then,
\begin{equation*}
\lim_{n \to \infty} \wk \set{
 \snorm{
\Wave_{n,M}\skl{\signal_n - \data_n}}_\infty \leq T \mkl{ \alpha_n, \sabs{\dict_{n,M}}}
\,;
\forall \signal_n   \in \R^{\I_n}
}  =  1 -  \alpha \,.
\end{equation*}
(Here by some abuse of notation  $\sabs{\dict_{n,M}}$
denotes the number of different elements in that frame, see \req{numel-cs}.)
The same  holds true if we replace
$T \mkl{ \alpha_n, \sabs{\dict_{n,M}}}$
by
\begin{equation}\label{eq:thresh-cs-alt}
T_{M} \skl{\alpha_n, n}
:=
- \sigma a\kl{\chi,n}\log \log \mkl{1/\skl{1-\alpha_n}}
+ \sigma b_M\kl{\chi,n} \,.
\end{equation}

\item  \label{it:cs3}\textbf{Smoothness:}
Let
$\hat \signal_{n,M}^*$ denote the soft-thresholding estimator using either $T \mkl{\alpha_n, \abs{\dict_n}}$ or $T_{n,M} \skl{\alpha_n}$ as threshold.
If the considered mother wavelet has $m>r$ vanishing moments and is $m$ times continuously differentiable,
then
\begin{equation*}
\liminf_{n\to\infty}
\wk\set{
\snorm{\hat\signal_{n,M}^*}_{\B_{p,q}^r}
\leq  \snorm{\signal}_{\B_{p,q}^r}
\,;
\forall
\signal  \in \B_{p,q}^r
} \geq   1 -  \alpha \,.
\end{equation*}
\end{enumerate}
\end{theorem}

\begin{proof}
By using the characterization of the  cycle spinning estimator in Lemma~\ref{lem:cs} and the cardinality computed in  Lemma~\ref{lem:cs-numel}, the  proof follows the lines of the proof of  Theorem~\ref{thm:wave}.
Again, one simply uses the fact that the discrete inner products coincide with continuous ones of functions forming an
infinite dimensional frame.
However, notice  the change  of the normalization sequences
in Item~\ref{it:cs1} which is also used for the threshold
$T_{n,M} \skl{\alpha_n}$.
As easy to verify  we have the asymptotic relation
\begin{multline*}
 a\kl{\chi,n}z
+ b_M\kl{\chi,n}
=
\sqrt{2 \log n}
+
\frac{2z-  \log\log n - \log\pi + 2 \log \log_2 \kl{M}}{2\sqrt{2\log n}}
\\
=
a\kl{\chi, \sabs{\dict_{n,M}}} z + b\kl{\chi, \sabs{\dict_{n,M}}}
+ o\kl{1/\sqrt{2\log n}}  \,.
\end{multline*}
From basic  extreme value theory it follows that we
can replace the sequence $a\skl{\chi, \sabs{\dict_{n,M}}} z + b\skl{\chi, \sabs{\dict_{n,M}}} $ by the one considered in Item~\ref{it:cs1} and for the threshold  $T_{n,M} \skl{\alpha_n}$.
\end{proof}

\begin{remark}
This alternative form   \req{thresh-cs-alt}  for the EVTs for
cycle spinning denoising  has  been introduced to allow a better comparison with the EVTs used in the basis case.
It fact, it can be seen that the extreme value thresholds $T_{M}\mkl{\alpha_{n}, n}$ for the redundant wavelet frame  $\dict_{n,M}$
simply increase by the  additive constant $\kl{\log\log_2 M} / \sqrt{2 \log n}$ when compared  to the extreme value threshold $T\kl{\alpha_{n},n}$ for the non-redundant wavelet frame.
\end{remark}

The sharp confidence regions of Theorem~\ref{thm:cs} require $M$ to be a  fixed number. In the fully translation invariant transform, to be discussed next, one takes $M = n$ dependent on the discretization level.
This effects a strong dependence of large scale coefficients
and that the distributional limit of Item~\ref{it:cs1} in
Theorem~\ref{thm:cs} will not longer hold true.

\subsubsection{Fully Translation Invariant Denoising}
\label{sec:ti}

Translation invariant wavelet denoising introduced in~\cite{CoiDon95,LanGuoOdeBurWel96,NasSil95,PesKriCar96}
is similar to cycle spinning denoising. However, now one
takes the whole range of $M = n$ integer shifts instead
of taking it as a fixed number independent on $n$.
That is, the translation  invariant  wavelet estimator is
defined by
\begin{equation}\label{eq:est-ti}
\hat \signal_{n,n}
:=
\frac{1}{n} \,
\sum_{m=0}^{n - 1}
\shiftop_{-m} \Wave_n^* \circ
\softop \kl{ \Wave_n \shiftop_m \data_n , T} \,.
\end{equation}

Lemma~\ref{lem:cs} implies that $\hat \signal_{n,n}$
equals   the soft-thresholding estimator $\Wave_{n,n}^+  \softop
\mkl{\Wave_{n,n} \data_n, T }$ where
$\Wave_{n,n}$ denotes the analysis operator corresponding to
the  translation  invariant wavelet frame
\begin{equation*}
\dict_{n,n}  =  \kl{\shiftop_m \wave_{j,k}^n: \skl{j,k} \in \Om_n \text{ and } m \in \set{0, \dots,  n - 1}} \,.
\end{equation*}
Equation~\req{numel-cs} shows that  the translation invariant wavelet frame contains $n \log_2 n$ different elements. Further, from the proof of this Lemma  it follows that
$\dict_{n,n}$ is a tight frame with  frame bound equals
$n$.  After removing  multiple elements in $\dict_{n,n}$, the resulting frame is non-tight but still has upper frame bounds
$b_n=n$ tending to infinity as $n \to \infty$.
One concludes that  Condition~\ref{it:frame3} fails to hold for the translation invariant wavelet transform.
The increasing  frame bounds somehow reflect the increasing
redundancy and dependency of the coarse scale wavelets with increasing
$n$.

One might conjecture that still  the  distribution result of Theorem \ref{thm:cs} holds true  with $M$
replaced by $n$. However, we shall show that this is not
the case.
Intuitively, the increasing correlation  of the coarse scale wavelets with increasing $n$ causes   the maximum
$\snorm{\Wave_{n,n}\eps_n}_\infty$ to be   in probability smaller than the maximum of  $n\log_2 n$
independent coefficients.
Although the sets if Theorem~\ref{thm:cs} are still confidence regions (as follows from Sidak's Lemma~\ref{lem:sidak}), they are no longer sharp
and the considered  thresholds are unnecessarily large.
The following theorem gives a much smaller radius for these confidence regions;
in particular this  significantly  improves~\cite[Theorem~4.4]{BerWel02}.

\begin{theorem}[Translation invariant soft-thresholding] \label{thm:ti}
Assume that the  mother wavelet $\wave$ is continuously differentiable, which implies that
$\skl{\bar\wave \ast \wave} \skl{t}
= 1 - c^2 t^2/2 + o\skl{t^2}$
for some constant $c$.
(Here  $\ast$ denotes the circular convolution and
$\bar\wave\kl{s} := \wave\kl{-s}$.)
Further denote by $\Wave_{n,n}$ the corresponding discrete translation invariant wavelet transform.

Then, the following assertions hold true:

\begin{enumerate}[label=(\alph*),topsep=0.0em]
\item \textbf{Distribution:}\label{it:ti1}
Let  $\skl{\noise_n}_{n\in \N}$ be a sequence of
noise vectors  in $\R^{\I_n}$ with independent $N\mkl{0,\sigma^2}$-distributed entries.
Then, for every $z\in \R$,
\begin{equation} \label{eq:gumbel-ti}
\liminf_{n \to \infty}
\wk\set{ \mnorm{\Wave_{n,n}\noise_n}_\infty
\leq
\sqrt{2 \log n}
+
\frac{z + \log \kl{ c/\pi}}{\sqrt{2\log n}}
}
\geq  \exp\kl{-e^{-z} }\,.
\end{equation}

\item \label{it:ti2}
\textbf{Confidence regions:}
Let $\alpha_n \in \kl{0,1}$
be a sequence  converging to some $\alpha \in [0,1)$.
Then, we have
\begin{equation*}
\liminf_{n \to \infty} \wk \set{
 \mnorm{\Wave_{n,n}
\mkl{\signal_n - \data_n} }_\infty \leq
T_{n} \mkl{\alpha_n}
\,;
\forall
\signal_n  \in \R^{\I_n}
}  \geq
1- \alpha  \,,
\end{equation*}
 when using the thresholds
$T_{n} \mkl{\alpha_n}
:=
\sigma \sqrt{2\log n}
+ \sigma \,
\skl{2\log n}^{-1/2} \log \log \skl{1/\skl{1-\alpha_n}} +  \log \kl{c / \pi}
$.

\item \label{it:ti3}
\textbf{Smoothness:}
Let $\hat \signal_{n,n}^*$ denote the soft-thresholding estimator using the threshold $T_{n}\skl{\alpha_n}$ defined  in
Item~\ref{it:ti2}. If the considered mother wavelet
has $m>r$ vanishing moments and is $m$ times continuously differentiable, then
\begin{equation*}
\liminf_{n\to\infty}
\wk\set{
\snorm{\hat\signal_{n,n}^*}_{\B_{p,q}^r}
\leq  \snorm{\signal}_{\B_{p,q}^r}  \,;
\forall
\signal  \in \B_{p,q}^r } \geq   1 -  \alpha \,.
\end{equation*}
\end{enumerate}
\end{theorem}

\begin{proof}
The key to all results is the distribution bound given
Item~\ref{it:ti1}. Its proof  is somehow  technical and is
presented in Section~\ref{ap:thm:ti}.
The   other claims follow  from Item~\ref{it:ti1} combined with the results of the previous sections (namely Theorems~\ref{thm:main},
\ref{thm:confidence} and~\ref{thm:smooth}), and
are verified as the corresponding  statements in the
proof of Theorem~\ref{thm:wave}.
\end{proof}

\begin{remark}
Consider  a sequence of standardized normal vectors $\eta_n$ each of them having  $ M n \kl{\log n}^r$ independent entries, where  $M$ is some fixed integer and  $r\geq0$    some fixed nonnegative number.
From   Proposition~\ref{prop:abs} we know
that $\snorm{\eta_n}_\infty$  is of Gumbel type with normalization sequences $a\kl{\chi, M n \kl{\log n}^r}$ and $b\kl{\chi, M n \kl{\log n}^r}$. One  easily verifies that
\begin{multline}\label{eq:evr}
a\kl{\chi, M n \kl{\log n}^r}z+ b\kl{\chi, M n \kl{\log n}^r}
\\
=
\sqrt{2\log n}
+
\frac{\kl{-1/2+r}\log\log n
+
\log \mkl{M / \sqrt{\pi}}}{\sqrt{2\log n}}
+
o\kl{1/\sqrt{2\log n} } \,.
\end{multline}
This allows to compare  the bound in \req{gumbel-ti}  with the asymptotic distribution of a certain number of independent random variables.
Indeed, comparing~\req{gumbel-ti} with \req{evr} we can conclude, that $\Wave_{n,n}\noise_n$ less or equal in probability than the maximum of
$ M n \sqrt{\log n}$ independent normally distributed random variables with $M:= [c \sqrt{\pi}]$.
Hence \req{gumbel-ti} improves  the primitive bound
obtained from the distribution of   $n\log n$  independent coefficients by a factor $\sqrt{\log n}/c$.\end{remark}

\begin{remark}
It is a difficult task to compute  the asymptotic
distribution of the translation invariant wavelet coefficients
exactly.
This is due to the fact that for coarse scales the
coefficients get increasingly correlated, whereas on the
fine scales the correlations remain bounded away from $\sigma^2$. No appropriate tools for asymptotic  extreme value analysis of  such mixed type random fields
seem to exist. Nevertheless, we  believe that the maxima  of the translation  invariant wavelet coefficients  are of Gumbel type but  with  even smaller  normalization constants than the ones used in~\req{gumbel-ti}.
In particular, it may even turn out that the threshold
$\sigma \sqrt{2 \log n}$  provides the denoising property for the translation
invariant system.\end{remark}

\subsection{Curvelet Thresholding}
\label{sec:curv}

Second generation curvelets
(introduced in  \cite{CanDon04,CanDon05a, CanDon05b})
are functions $\wave_{j,\ell,k}$ in $\R^2$ depending on a scale index $j \in \N$,
an orientation parameter  $\ell \in \sset{0, \dots ,  4 \cdot 2^{\ckl{j/2}}-1}$ and a location parameter  $k\in \Z^2$.
They are known to provide an almost optimal sparse approximation of piecewise $C^2$ functions with piecewise $C^2$ boundaries  (as shown in~\cite{CanDon04});  this class of  functions usually  serves as accurate cartoon model for natural images.
The main curvelet property yielding this  approximation
result is the increasing  anisotropy at finer scales. This  feature also distinguishes them from standard wavelets in
higher dimension.

There exists other related function systems with
similar properties.
The  cone adapted shearlet frame
(introduced in \cite{GuoKutLab06,GuoLabLimWeiWil06,LabLimKutWei05})
is  very similar to the curvelet frame and  shares its
optimality when  approximating piecewise $C^2$ images
with piecewise $C^2$ boundaries, see~\cite{GuoLab07}.
Yet another closely related function system are the
contourlets introduced by Do and Vetterli~\cite{DoVet03,DoVet05}.
For simplicity we focus  on the curvelets; similar statements could be made for the shearlet and contourlet frames.

\subsubsection{Discrete Curvelet Frames}

The discrete  curvelet transform  computes inner products of
$\signal_n \in \R^{n\times n}$ with discrete curvelets
$\wave_{j,\ell,k}^n \in \R^{n\times n}$.
As for the wavelet transform,  the  elements
$\wave_{j,\ell,k}^n$ are not computed explicitly and
defined implicitly by the transform algorithm.
Different  implementations  of the continuous curvelet transform give rise to different discrete frame elements
$\wave_{j,\ell,k}^n$. Current implementations of the curvelet
transform are  computed in the Fourier domain.
Below we shall focus on the wrapping based implementation
of the curvelet transform introduced in~\cite{CanDemDonYin06}.
This transform is an isometry which makes the computation of its
pseudo-inverse particularly simple.

Let $n =  2^J$ be an integer power of two with $J$ denoting the maximal scale index.
The discrete  curvelets and the discrete curvelet transform
are composed of the following ingredients:
\begin{itemize}[topsep=0em]
\item
First, define  $\Lambda_n$ as the set of all
pairs $\skl{j,\ell}$ satisfying    $j \in \set{0, \dots, \log_2 n-2 }$ and
$\ell \in \sset{0, \dots ,  4 \cdot 2^{\ckl{j/2}}-1}$.
The  index sets of the discrete curvelets is defined by
\begin{equation*}
\Om_n:=
	\bigl(\kl{j,\ell,k}:\skl{j,\ell}\in \Lambda_n \text{ and }
	k \in D_{j,\ell} \bigr) \,.
\end{equation*}
where $D_{j,\ell} = \sset{ 0, \dots, K_{j,\ell;1}-1} \times \sset{ 0 , \dots, K_{j,\ell;2}-1}$
for certain given numbers $K_{j,\ell;1}\sim 2^{j}$ and $K_{j,\ell;2}\sim 2^{j/2}$.
One refers to $j$ as scale index, $\ell$ as the orientation index, and $k$ the location index.

\item
Next, one constructs smooth nonnegative window functions
$\wi_{j,\ell}\colon \R^2 \to \R$  satisfying the identity
\begin{equation*}
	\kl{\forall z \in \R^2}\qquad
	\sum_{j=0}^{J-2} \sum_{\ell=0}^{4 \cdot 2^{\ckl{j/2}}-1}
	\sabs{\wi_{j,\ell}\skl{z}}^2 =1 \,.
\end{equation*}
The functions $\wi_{j,\ell}$ are  essentially obtained  by anisotropic scaling and shearing a single window function;
see~\cite{CanDemDonYin06} for a detailed construction.

\item
For any  index triple $\kl{j,\ell,k} \in \Om_n$
the  discrete curvelet at  scale $j$, having
orientation $\ell/2^{\ckl{j/2}}$
and  location $k = \skl{k_1/K_{j,\ell;1}, k_2/K_{j,\ell;2}}$ is
defined by its Fourier representation
\begin{equation}\label{eq:curv-def}
\mkl{\Ft_n \curve_{j,\ell,k}^{n}}
\skl{m}
=
\frac{\wi_{j,\ell}\kl{m} }{c_{j,\ell}}
e^{- 2\pi i
\skl{m_1k_1/K_{j,\ell;1}-m_2 k_2/K_{j,\ell;2}}} \,.
\end{equation}
Here the coefficients $c_{j,\ell}$ are chosen   in such a way that $\snorm{\wave_{j,\ell,k}^n} = 1$
and $\Ft_n$ denotes the discrete Fourier transform.

\item
Finally, one defines  the curvelet frame $\dict_n= \mkl{\wave_{j,\ell,k}^n: \kl{j,\ell,k} \in \Om_n}$ and denotes
by $\Ct_n\colon \R^{n\times n} \to \R^{\Om_n}$ the
corresponding analysis operator,  which has been named
digital curvelet transform via wrapping
in \cite{CanDemDonYin06}.  In the following we will  refer to
$\Ct_n$ simply as \emph{discrete curvelet transform}.
We emphasize again that the implementation of the discrete curvelet transform does not require to compute the
curvelets   $\wave_{j,\ell,k}^n$ explicitly.
\end{itemize}

Implementations of the discrete  curvelet transform and its  pseudoinverse using $\mathcal O \skl{n^2 \log n }$ operation counts are freely available
at~\url{http://curvelet.org}.
Although this  implementation  does not use  normalized
frame elements, the   constants $c_{j,\ell}$ can easily be computed after  the actual curvelet transformation and applied for normalizing the curvelet coefficients prior to
denoising.  The  denoising
demo  \verb"fdct_wrapping_demo_denoise.m" included in the curvelet  software package in fact computes  the norms of the discrete curvelets and uses them  for proper scaling
of the chosen thresholds.

 \subsubsection{Curvelet Denoising}

We now consider our denoising problem~\req{problem}, which, after taking the discrete curvelet transform,
simply reads   $\Data_n = \Para_n +
\Ct_n \noise_n$.
As usual, the estimator we consider is soft-thresholding $\hat\Para_n = \softop \kl{\Data_n,T}$ of the curvelet coefficients.

Similar to the wavelet case, we  measure smoothness in terms of the weighted $(p,q)$-norms, depending on
certain norm parameters $p,q \geq 1$ and a parameter $r\geq 0$ describing the degree of  smoothness.
More precisely, we define
\begin{equation*}
\norm{\Para_n}_{p,q,r}
:=
\sqrt[p]{\sum_{j,\ell}
2^{jsq}
\norm{\Para_n\skl{j,\ell,\cdot}}_p^q}
\qquad
\text{ with } \quad
s = r + \frac{3}{2} \kl{\frac{1}{2}- \frac{1}{p}} \,.
\end{equation*}
These types of norms applied to the continuous domain curvelet coefficients have been defined and studied  in~\cite{BorNil07}. In that paper also relations between  these norms and classical Besov norms have been derived.

\begin{theorem}[Curvelet soft-thresholding] \label{thm:curve}
The discrete curvelet frames
$\dict_n = \mkl{\wave_{j,\ell,k}^n: \kl{j,\ell,k} \in \Om_n}$
defined by Equation~\req{curv-def} are asymptotically stable.
In particular, the following assertions hold true:

\begin{enumerate}[label=(\alph*),topsep=0.0em]
\item\label{it:curve1}
\textbf{Distribution:}
Let  $\noise_n$ be a sequence of
noise vectors  in $\R^{\I_n}$ with independent
$N\mkl{0,\sigma^2}$-distributed entries.
Then, for   every $z\in \R$,
\begin{equation*}
\lim_{n \to \infty}
\wk\set{ \norm{\Ct_{n}\noise_n}_\infty
\leq
\sigma \sqrt{2 \log \abs{\Om_n}}
+
\sigma
\, \frac{2z-  \log\log \abs{\Om_n} - \log\pi }{2\sqrt{2\log \abs{\Om_n}}}
}
= \exp\kl{-e^{-z} }\,.
\end{equation*}

\item \label{it:curve2}
\textbf{Confidence regions:}
Let $\alpha_n \in \kl{0,1}$
be a sequence  of significance levels converging to some
$\alpha \in [0,1)$ and let
$T \mkl{\alpha_n, \abs{\Om_n}}$ denote the
corresponding  EVTs defined in~\req{alphathresh}.
Then,
\begin{equation*}
\lim_{n \to \infty} \wk \set{
 \norm{
\Ct_n\kl{\signal_n - \data_n}}_\infty \leq T \kl{\alpha_n, \abs{\Om_n}}
\,;
\signal_n   \in \R^{\I_n}
}  =  1 -  \alpha \,,
\end{equation*}

\item \label{it:curve3}
\textbf{Smoothness:}
Let $\hat\Para_n$ denote the soft-thresholding
estimate  using the extreme value threshold $T \mkl{\alpha_n, \abs{\Om_n}}$.
Then, with any of the  norms defined above, we have
\begin{equation*}
\liminf_{n \to \infty} \wk
\set{
\snorm{\hat\Para_n}_{p,q,r}
\leq  \snorm{\Ct_n \signal_n}_{p,q,r}
\,;
\signal_n   \in \R^{\I_n}
} \geq
1 -  \alpha \,.
\end{equation*}
\end{enumerate}
\end{theorem}

\begin{proof}
All frame elements are normalized due to the chosen scaling. Moreover, as shown in~\cite[Proposition~6.1]{CanDemDonYin06}, the
discrete curvelet frame $\dict_n$ is faithful to an underlying infinite dimensional  curvelet frame  obtained by periodizing the curvelets on
the continuous domain $\R^2$.  This immediately yields Condition~\ref{it:frame3}.
Moreover, along the lines of~\cite{CanDon05a} (which uses a slightly different curvelet system) one easily shows that the inner products satisfy
\begin{equation*}
\inner{\wave_{j,\ell,k}^n}{\wave_{j',\ell',k'}^{n}} \leq \rho < 1 \,.
\end{equation*}
for some constant $\rho < 1$ independent on $n$ and all indices. This obviously implies Condition~\ref{it:frame2}.
All  claims in Items~\ref{it:curve1}--\ref{it:curve3}
then follow  from Theorems~\ref{thm:main},  \ref{thm:confidence} and   \ref{thm:smooth}
derived in the previous section.
\end{proof}

\section*{Acknowledgement}

Both authors want to thank Daphne B\"ocker, Helene Dallmann, Housen Li and Andrea Kraijna for useful discussions and comments on the manuscript.
The work of A. M.  has been supported by the Volkswagen Foundation and the
 German Science Foundation (DFG) within the projects FOR 916, CRC 755, 803 and RTN 1023.

\appendix

\section{Extremes of Normal Random Vectors}
\label{ap:evd}

Let $\skl{\Om_n}_{n \in \N}$  be  a sequence of  finite index  sets
with monotonically increasing cardinalities
$\abs{\Om_n}$  satisfying
$\lim_{n \to \infty} \abs{\Om_n}   =  \infty$.
Moreover, for every $n \in \N$, let  $\Noise_n  := \skl{\Noise_{n}\skl{\om}: \om \in \Om_n}$ be given
standardized  normal random vectors, which means that  $\Noise_{n}\skl{\om} \sim N\mkl{0,1}$ for every  $n \in \N$ and
$\om \in \Om_n$.  We  are mainly interested in  random vectors with dependent entries, in which  case $\cov \skl{\Noise_n\kl{\om} , \Noise_n\kl{\om'} } \neq 0$ for at least some pairs $\skl{\om , \om'} \in \Om_n^2$ with  $\om \neq \om'$.

As the main result of this section we derive the asymptotic  distribution of $\snorm{\Noise_n}_\infty$  for a sequence  $\Noise_n$ of  dependent normal  vectors whose covariances satisfy a certain summability condition (see Theorem~\ref{thm:abs}).
Whereas similar results are  known for $\max\mkl{\Noise_n}$,
to the best of our knowledge,  such kind of results
are new for $\snorm{\Noise_n}_\infty$.
Since   $\sabs{\Noise_{n}\skl{\om}}^2 \sim \chi^2$ is  chi-squared   distributed with one degree of freedom,  our
results can also be interpreted as new results for
the asymptotic extreme value theory of
dependent $\chi^2$-distributed random vectors.

\subsection{Maxima of Normal Vectors}

We will start by reviewing and slightly refining  the main results from  statistical  extreme value theory for maxima of normal vectors as we require them in this paper.

The most basic extreme value result deals with the case   where  the components of
$\Noise_n$ are independent. In this case it is well known,
that, after rescaling, $\max\mkl{\Noise_n}$ converges to the  Gumbel distribution as $n \to \infty$.

\begin{proposition} \label{prop:normal}
Let $\mkl{\Noise_n}_{n\in\N}$ be a sequence of standardized normal  random vectors in $\R^{\Om_n}$
with independent entries. Then the maxima $\max\mkl{\Noise_n}$
are of Gumbel type (see Definition~\ref{def:gumbeltype}) with normalization sequences  $a\skl{\sabs{\Om_n}, n}$, $b\skl{\sabs{\Om_n}, n}$ defined
by  \req{am}, \req{bm}.
\end{proposition}

\begin{proof}
See, \cite[Theorem~1.5.3]{LeaLinRoo83}.
\end{proof}

If the  entries of $ \Noise_n$ are dependent,
then the result of Proposition~\ref{prop:normal} does not necessarily hold true. There is, however, a simple  and sufficient criterion  on the covariances $\cov \mkl{\Noise_n\skl{\om}, \Noise_n\skl{\om'}}$ of a sequence of  dependent normal vectors
such that the maxima still are of Gumbel type with
the same normalization sequences. This criterion is an immediate consequence of the so called normal comparison  Lemma  or Berman's inequality (see \cite[Theorem~4.2.1]{LeaLinRoo83}).
For later purpose, where we  study   $\snorm{\Noise_n}_\infty$ instead of   $\max\mkl{\Noise_n}$, we require  a quite
recent improvement of this  important inequality which is
due to Li and Shao~\cite{LiSha02}.
The  standard form of the  normal comparison  Lemma \cite[Theorem~4.2.1]{LeaLinRoo83} has  already been applied
for redundant wavelet systems in~\cite{BerWel02,NgKohPuvAbe08}, which however,
only yields results for  maxima of $\Noise_n$
without taking absolute values.
We stress again, that taking  absolute values slightly  change the constants  in contrast to
relations \req{noise-remove}.

\begin{lemma}\label{lem:comparison}
Let $\eta_n$, $\Noise_n$  be standardized normal
random vectors in $\R^{\Om_n}$, denote its covariances  by
 $\kappa_{\eta_n}\skl
{\om,\om'} :=
\cov \mkl{\eta_n\skl{\om}, \eta_n \skl{\om'} }$,
$ \kappa_{\Noise_n}\skl{\om,\om'} :=
\cov \mkl{\Noise_n\kl{\om}, \Noise_n\skl{\om'}}$,
and set  $\rho_n\skl{\om,\om'} := \max \mset{\sabs{\kappa_{\eta_n}\skl{\om,\om'}}, \sabs{\kappa_{\Noise_n}\skl{\om,\om'}}}$.
Then, for all $T_n \in \R$,
\begin{multline*}
\wk \set{ \max \mkl{\eta_n} \leq T_n }  -
\wk \set{ \max \mkl{\Noise_n}   \leq T_n}
\leq
\\
\frac{1}{4\pi}
\sum_{\om \neq \om' }
 \bkl{\arcsin\mkl{\kappa_{\eta_n}\skl{\om,\om'}} -
 \arcsin\mkl{\kappa_{\Noise_n}\skl{\om,\om'}}}_+   \;
\exp\kl{ \frac{-T_n^2} {1+\rho_n\skl{\om,\om'}}} \,.
\end{multline*}
Here    $z_+= \max \set{z,0}$  denotes   the positive part of
some   real number $z\in \R$ and $\arcsin$ denotes the inverse mapping of $\sin\colon  \ekl{-\pi/2, \pi/2} \to \ekl{-1, 1}$.
\end{lemma}

\begin{proof}
See~\cite[Theorem~2.1]{LiSha02}.
\end{proof}

In the special case where $\eta_n$  has
independent entries,  Lemma~\ref{lem:comparison}
has the following  immediate consequence given
in Lemma~\ref{lem:normal}.
This allows to extend Proposition~\ref{prop:normal}
to certain  sequences of dependent random vectors
by comparing them with  independent ones.

\begin{lemma}\label{lem:normal}
Let $\eta_n, \Noise_n$ be standardized normal random vectors in $\R^{\Om_n}$.
Assume that the entries  of $\eta_n$ are independent, and let
$\kappa_{n}$ denote  the covariance matrix of $\Noise_n$ defined by
 $\kappa_{n} \skl{\om,\om'} := \cov \mkl{ \Noise_n\kl{\om}, \Noise_n\kl{\om'}}$.
Then, for  all $T_n \in\R$,
\begin{multline} \label{eq:berman-norm}
\abs{\wk \set{  \max\mkl{\eta_n}    \leq T_n} -
\wk \set{ \max\mkl{\Noise_n}    \leq T_n} }
\\ \leq
\frac{1}{8}
\sum_{\om \neq \om'}
\abs{\kappa_{n} \mkl{\om,\om'}}
 \,
\exp\kl{- \frac{T_n^2}{ 1+\sabs{\kappa_{n}\skl{\om,\om'}}}}  \,.
\end{multline}
\end{lemma}

\begin{proof}
See~\cite[Corollary~2.2]{LiSha02}.
\end{proof}

Lemmas~\ref{lem:comparison} and~\ref{lem:normal}
are  significant  improvements of the standard versions
of the normal  comparison Lemma~\cite[Section~4]{LeaLinRoo83}
due to the absence of a singular factor
$\mkl{1- \sabs{\rho_n\skl{\om, \om'}}^2}^{-1/2}$  that is contained in earlier versions.
It is in fact the absence  of this  singular term that
we require  for deriving an inequality similar to the one in
Lemma~\ref{lem:normal} that compares the distributions of
$\snorm{\eta_n}_\infty$ and  $\snorm{\Noise_n}_\infty$
for two normal vectors  $\eta_n$ and $\Noise_n$.

\begin{theorem}\label{thm:normal}
Let $\mkl{\Noise_n}_{n \in \N}$ be a sequence of standardized normal  random vectors in $\R^{\Om_n}$ having covariance matrices
$\kappa_{n} \in \R^{\Om_n \times \Om_n}$
satisfying
\begin{equation}\label{eq:rest}
\lim_{n \to \infty}
\sum_{\om \neq \om'}
\abs{\kappa_{n} \mkl{\om,\om'}}
 \,
\kl{ \frac{\log \abs{\Om_n}}{\abs{\Om_n}^2}}^{1/\kl{1+\sabs{\kappa_n\skl{\om,\om'}}}}
=0  \,.
\end{equation}
Then, the maxima  $\max\mkl{\Noise_n}$ are of Gumbel type (see Definition~\ref{def:gumbeltype}) with
normalization constants $a\kl{N, \abs{\Om_n}}$, $b\kl{N, \abs{\Om_n}}$ defined by   \req{am}, \req{bm}.
\end{theorem}

\begin{proof}
Fix some $z \in \R$ and define   $T_n  := a\skl{N,\sabs{\Om_n}}z+ b\skl{N,\sabs{\Om_n}}$.
Then, the definitions of  the normalization sequences
$a\skl{N,\sabs{\Om_n}}$ and $b\skl{N,\sabs{\Om_n}}$ imply that $T_n^2 =  2 \log \abs{\Om_n} - \log \log \abs{\Om_n} + \mathcal O\kl{1} $ as $n \to \infty$.  Hence there
is some constant  $C >0$ and some index $n_0\in \N$,
such that for all $n \geq n_0$, we have
\begin{equation*}
\exp\kl{-\frac{T_n^2}{1+ \sabs{\kappa_{n}\skl{\om, \om'}}}}
\leq
C
\exp\kl{-\frac{\log \mkl{\abs{\Om_n}^2/\log\abs{\Om_n}}}{1+ \sabs{\kappa_n\skl{\om, \om'}}}}
\\
= C
\kl{ \frac{\log \abs{\Om_n}}{\abs{\Om_n}^2}}^{1/\kl{1+\sabs{\kappa_n\skl{\om,\om'}}}} \,.
\end{equation*}
Now let Equation~\req{rest} be satisfied and let $\skl{\eta_n}_{n \in \N}$ be a sequence of standardized normal vectors with independent entries.
Then, the  triangle inequality,  Lemma~\ref{lem:normal},
and the estimate  just established  imply
\begin{equation*}
\mabs{\wk \set{ \max\mkl{\Noise_n}   \leq T_n}}
\leq
\mabs{\wk \set{ \max\kl{\eta_n}   \leq T_n}}
\\
+
\frac{C}{8}
\sum_{\om \neq \om' }
 \abs{\kappa_{n}\skl{\om,\om'}}   \;
\kl{ \frac{\log \abs{\Om_n}}{\abs{\Om_n}^2}}^{1/\kl{1+\sabs{\kappa_n\skl{\om,\om'}}}}
\,.
\end{equation*}
Hence the claim  follows from  Proposition~\ref{prop:normal}
and  Assumption~\req{rest}.
\end{proof}

\subsection{Maxima of Absolute Values}

In the following we derive results
similar to Proposition~\ref{prop:normal} and   Theorem~\ref{thm:normal}
for $\snorm{\Noise_n}_\infty$ in place of $\max \mkl{\Noise_n}$.
The first auxiliary result, Proposition~\ref{prop:abs},
deals with the independent case. It is easy to establish  but nevertheless seems  to be much less  known   than the corresponding result in the normal case. We include a short proof based on the known extreme value distribution of  independent $\chi^2$-distributed random variables.
The second and main  result in this section, Theorem~\ref{thm:abs},  deals with the dependent case.
It is a new contribution and based on a novel   inequality  for comparing the distributions of
$\snorm{\eta_n}_\infty$ and
$\snorm{\Noise_n}_\infty$ (given in Lemma~\ref{lem:abs}).

\begin{proposition} \label{prop:abs}
Let $\mkl{\Noise_n}_{n \in \N}$ be a sequence of  standardized normal  vectors in $\R^{\Om_n}$ having independent entries.
Then $\norm{\Noise_n}_\infty$ is of Gumbel type (see Definition~\ref{def:gumbeltype}) with normalization sequences $a\kl{\chi, \abs{\Om_n}}$, $b\kl{\chi, \abs{\Om_n}}$ defined by   \req{a-abs}, \req{b-abs}.
\end{proposition}

\begin{proof}
Since  $\Noise_{n}\skl{\om}$ is  standard normally distributed for any
$\om \in \Om_n$, the random variables
$\sabs{\Noise_{n}\skl{\om}}^2$ are  $\chi^2$-distributed with one degree of freedom. The $\chi^2$-distribution  is in turn  a member of the family of    Gamma distributions  $F_{\beta,\gamma}$ corresponding to   $\beta = \gamma = 1/2$.
The asymptotic extreme value distribution  of the Gamma distribution  $F_{\beta,\gamma}$ is known
(see~\cite[page 156]{EmbKluMik97})  and implies
\begin{equation} \label{eq:gumbel-chi}
\lim_{n \to \infty} \wk \set{ \snorm{\Noise_n}_\infty^2
\leq   2z + 2 \log \abs{\Om_n}  - \log\log \abs{\Om_n}  - \log \pi }
\\= \exp\kl{-e^{-z} } \,.
\end{equation}
Moreover, a Taylor series  approximation shows
 \begin{multline}\label{eq:taylor-1}
\sqrt{ 2z + 2 \log \abs{\Om_n}  - \log\log \abs{\Om_n}  - \log \pi}
\\=
a\mkl{\chi,\abs{\Om_n}}z + b\mkl{\chi,\abs{\Om_n}}
+
o\kl{ a\mkl{\chi,\abs{\Om_n}} }
\quad \text{ as } n \to \infty \,.
\end{multline}
Any  $o\mkl{ a\mkl{\chi,\abs{\Om_n}} }$-term can be omitted
when computing extreme value distributions
(see \cite[Theorem~1.2.3]{LeaLinRoo83}), and hence
Equations~\req{gumbel-chi} and~\req{taylor-1} imply the desired
result.
\end{proof}

\begin{remark}\label{rem:normalizing}
The   sequence  $b\mkl{\chi,\abs{\Om_n}} $
used for normalizing  the maximum
$\snorm{\Noise_n}_\infty$ in  Proposition~\ref{prop:abs} is
different from the sequence $b\mkl{N,\abs{\Om_n}}$
used for the  normalization of   $\max\mkl{\Noise_n}$
in Proposition~\ref{prop:normal}.
Indeed, as easily verified,
\begin{equation*}
b\mkl{N, 2 \abs{\Om_n} }
=
b\mkl{\chi, \abs{\Om_n} } +
o\kl{ a\mkl{N,2\abs{\Om_n}} }
\,.\end{equation*}
Again,  the $o\mkl{ a\mkl{N,2\abs{\Om_n}} }$ term can be omitted
in the extreme value distribution and hence
$\snorm{\Noise_n}_\infty$ behaves equal to the  maximum of
$2 \abs{\Om_n}$ (opposed to $\abs{\Om_n}$)
independent standard normally  distributed random  variables.
Using different arguments, this has already been observed
in~\cite[Section~8.3]{Joh11}.\end{remark}

If the entries of $\Noise_n$ are  not independent,
then the result of  Proposition~\ref{prop:abs} does not
necessarily hold true.
If, however,    the correlations   of $\Noise_n$ are sufficiently small,   then, as in the normal case,  we will show  that the same Gumbel law still holds.
This result follows again from  a comparison inequality, now   between the  distributions of $\snorm{\Noise_n}_\infty$
and $\snorm{\eta_n}_\infty$ with some  reference normal vector
$\eta_n$,  to be derived in the following  Lemma~\ref{lem:abs}.
For the sake of simplicity we assume that the vector $\eta_n$
has independent entries; in an analogous manner a similar result could be
derived for comparing two dependent random vectors.

\begin{lemma}\label{lem:abs}
Let $\eta_n, \Noise_n$ be standardized normal  random vectors in $\R^{\Om_n}$.
Assume that the entries  of $\eta_n$ are independent and denote by
$\kappa_{n} \in \R^{\Om_n \times \Om_n}$ the covariance matrix  of $\Noise_n$, having   entries
$\kappa_{n}\skl{\om,\om'} := \cov \mkl{ \Noise_n\kl{\om}, \Noise_n\kl{\om'}}$.
Then, for  all $T_n \in\R$,
\begin{equation} \label{eq:berman-abs}
\abs{\wk \set{  \snorm{\eta_n}_\infty   \leq T_n} -
\wk \set{ \snorm{\Noise_n}_\infty   \leq T_n} }
\leq
\frac{1}{4}
\sum_{\om \neq \om'}
\abs{\kappa_n \skl{\om,\om'}}
 \,
\exp\kl{- \frac{T_n^2}{ 1+\sabs{\kappa_{n}\skl{\om,\om'}}}}  \,.
\end{equation}
\end{lemma}

\begin{proof}
The proof  uses the normal comparison  Lemma~\ref{lem:comparison} of Li and Shao  applied to  the strongly
dependent random vectors  $Y_n :=  \mkl{\eta_n, -\eta_n}$
and $X_n :=  \mkl{\Noise_n, -\Noise_n}$
in place of $\eta_n$ and  $\Noise_n$.
To that end, we first note that obviously
$\mset{ \sabs{\Noise_n}  \leq T_n} = \mset{ X_n  < T_n}$  and
$ \mset{ \sabs{\eta_n}   \leq T_n} =  \mset{Y_n < T_n}$.
Moreover, the covariance matrices of $Y_n$ and
$X_n$ are block matrices of the form
\begin{equation*}
\cov\skl{Y_n }
=
\begin{pmatrix}
\phantom{-} \mathbf{I}_n &  - \mathbf{I}_n \\
-\mathbf{I}_n &   \phantom{-} \mathbf{I}_n
\end{pmatrix}
\quad \text{ and } \quad
\cov\skl{X_n}
=
\begin{pmatrix}
\phantom{-} \kappa_n &  - \kappa_n \\
-\kappa_n  &   \phantom{-} \kappa_n
\end{pmatrix}
\,,
\end{equation*}
where $\kappa_n = \cov\mkl{\Noise_n}$ denotes the covariance matrix of $\Noise_n$ and $\mathbf{I}_n = \cov\mkl{Y_n}$ is the identity matrix in $\R^{\Om_n \times \Om_n}$.
Now applying Lemma \ref{lem:comparison} with
$Y_n$ and
$X_n$ in place of  $ \eta_n$ and $\Noise_n$ yields
\begin{multline*}
\wk \set{ \norm{\eta_n}_\infty   \leq T_n }
-
\wk \set{ \norm{\Noise_n}_\infty  \leq T_n }
\\
\begin{aligned}
=&
\wk \set{ \max  \kl{\eta_n, -\eta_n} \leq T_n }
-
\wk \set{ \max  \kl{\Noise_n, -\Noise_n}   \leq T_n }
\\
\leq
&
\frac{1}{2\pi}
\sum_{\om \neq \om'}
\kl{ \skl{-\arcsin\skl{\kappa_n\skl{\om,\om'}}}_+
+ \skl{\arcsin\skl{\kappa_n\skl{\om,\om'}}}_+ }
 \exp\kl{- \frac{T_n^2}{1+\sabs{\kappa_n\skl{\om,\om'}}}}
\\
=&
\frac{1}{2\pi}\sum_{\om \neq \om'}
\sabs{\arcsin\skl{-\kappa_n\skl{\om,\om'}}} \;
\exp\kl{- \frac{T_n^2}{1+\sabs{\kappa_n\skl{\om,\om'}}}}
\,.
\end{aligned}
\end{multline*}
Here for the  first estimate we used that the two sums over the diagonal blocks   give the same value,  that the same is the case for the two off-diagonal blocks,
that all terms having   $\om = \om'$ cancel and
that  $\cov\skl{X_n}\kl{\om, \om'} = 0$  for  $\om \neq  \om'$.
Interchanging the roles of  $\Noise_n$ and $\eta_n$ yields the same estimate for
 $\wk \mset{  \snorm{\Noise_n}_\infty \leq T_n} - \wk \mset{ \snorm{\eta_n}_\infty \leq T_n}$  and hence implies
 \begin{multline}   \label{eq:berman-abs-asin}
\abs{\wk \set{  \norm{\eta_n}_\infty   \leq T_n } -
\wk \set{ \norm{\Noise_n}_\infty   \leq T_n } }
\\\leq
\frac{1}{2\pi}\sum_{\om \neq \om'}
\mabs{\arcsin\skl{\kappa_n\skl{\om,\om'}}} \;
\exp\kl{- \frac{T_n^2}{1+\sabs{\kappa_n\skl{\om,\om'}}}} \,.
\end{multline}
Finally, the estimate  $\xsabs{\arcsin y } \leq \xsabs{y} \cdot \pi / 2$
for  $y \in \ekl{-1,1}$ and inequality~\req{berman-abs-asin} imply the claimed inequality~\req{berman-abs}.
\end{proof}

The following theorem is the main result of this section and the key for most results established in this paper.

\begin{theorem}\label{thm:abs}
Let $\mkl{\Noise_n}_{n \in \N}$ be a sequence of standardized normal vectors in $\R^{\Om_n}$ having covariance matrices   $\kappa_{n} \in \R^{\Om_n \times \Om_n}$   satisfying Equation~\req{rest}.
Then $\norm{\Noise_n}_\infty$ is of Gumbel type (see Definition~\ref{def:gumbeltype}) with normalization constants $a\kl{\chi, \abs{\Om_n}}$, $b\kl{\chi, \abs{\Om_n}}$ defined by   \req{a-abs}, \req{b-abs}.
\end{theorem}

\begin{proof}
This is analogous to the proof of Theorem~\ref{thm:normal}.
Instead of  Proposition~\ref{prop:normal}  and
Lemma~\ref{lem:normal} one now
uses  Proposition~\ref{prop:abs}  and
Lemma~\ref{lem:abs}.
\end{proof}

Equation~\req{rest} provides a sufficient condition
for the extreme value results of
Theorems~\ref{thm:normal}  and~\ref{thm:abs} to hold.
However, given a sequence $\mkl{\Noise_n}_{n \in \N}$ of
normal vectors with covariance matrices $\kappa_n$, it is  not completely obvious  whether or not~\req{rest} is satisfied.
In Section~\ref{sec:evd-frame} we verified
that~\req{rest} indeed holds in the case where
$\Noise_n =  \sinner{\base_{\om}^{n}}{\noise_n}$ are coefficients of  standardized  normal random vectors  $\noise_n$ having independent entries with respect to an asymptotically stable family
of frames $\mkl{\base_{\om}^{n}: \om\in\Om_n}$.

Occasionally we will make use of the following classical
result due to Sidak~\cite{Sid67} for bounding the
maximum of the magnitudes of  dependent random vectors by the maximum of the magnitudes of independent ones.

\begin{lemma}[Sidak's Inequality]\label{lem:sidak}
Let $\eta_n, \Noise_n$ be standardized normal  random vectors in $\R^{\Om_n}$ and assume  that the entries  of $\eta_n$ are independent.
Then,
\begin{equation} \label{eq:sidak}
\kl{\forall T  \in\R}
\qquad
\wk\set{  \norm{\Noise_n}_\infty   \leq T  }
\geq
\wk\set{ \norm{\eta_n}_\infty   \leq T }  \,.
\end{equation}
\end{lemma}

\begin{proof}
See \cite[Corollary~1]{Sid67}.
\end{proof}

Note that a similar result also holds for the
maxima without the absolute values,
which bounds the probability
$\wk\mset{  \max \mkl{\Noise_n}   \leq T  } $ of dependent standardized normal vectors from below by the probability $\wk\mset{  \max\skl{\eta_n}   \leq T } $
of independent ones. This one-sided estimate,  however, requires the covariances of $\Noise_n$ being nonnegative. It is  known as  Slepian's Lemma
and has first been derived in~\cite{Sle62}.
Interestingly, Slepian's Lemma immediately follows from the normal comparison Lemma~\ref{lem:comparison}, whereas this seems not to
be the case for Sidak's two sided inequality.

\section{Remaining Proofs}

\subsection{Proof of Proposition~\ref{prop:risk}}
 \label{ap:risk}
As already noted in~\cite[page~558]{Mal09},
for the universal  thresholds $\sigma \sqrt{2 \log \abs{\Om_n}}$
this result  easily follows by adapting the original proof
of ~\cite{DonJoh94} (see also \cite[Section~8.3]{Joh11} and
\cite[Theorem~11.7]{Mal09})
from the orthonormal  case to  the frame case.
Indeed,  as shown below a  similar proof can be
made for the extreme value thresholds $T_n
=T\mkl{\alpha_n, \abs{\Om_{n}}}$ defined by Equation  \req{alphathresh}.

After rescaling we  may assume  without
loss of generality that $\sigma =1$.
Recall that the dual frame
$\mkl{\dbase_{\om}^{n}: \om\in\Om_n}$
 has upper frame bound $1/a_n$,
 that $\Base_n^+ \Base_n = \operatorname{Id}$ is the identity on $\R^{\I_n}$,
and that $\Base_n \Base_n^+ = P_{\range\skl{\Base_n}}$
equals the orthogonal projection onto the range $\range\skl{\Base_n} \subset \R^{\Om_n}$ of
the analysis operator $\Base_n \colon \R^{\I_n} \to
\R^{\Om_n}$. Moreover, we define the parameter $\Para_n =\Base_n\signal_n$ and the
data $\Data_n = \Base_n\data_n$
as in~\req{DataPara}.
Then we can estimate
\begin{multline*}\hspace{3em}
\ew \kl{\mnorm{\signal_n -  \Base_n^+
\circ \softop \kl{\Base_n\data_n,
T\mkl{\alpha_n, \abs{\Om_{n}}}}}^2 }
\\
\begin{aligned}
& =
\ew \kl{\mnorm{\Base_n^+ \Para_n -  \Base_n^+
\circ \softop \kl{\Base_n\data_n,
T\mkl{\alpha_n, \abs{\Om_{n}}}}}^2 }
\\
& =
\ew \kl{\mnorm{\Base_n^+ \Base_n \Base_n^+ \kl{\Para_n - \softop \kl{ \Data_n, T\mkl{\alpha_n, \abs{\Om_{n}}}}}}^2 }
\\
& \leq
\frac{1}{a_n}
\ew \kl{\mnorm{\mathrm{P}_{\range\skl{\Phi_n}} \kl{\Para_n - \softop \kl{ \Data_n, T\mkl{\alpha_n, \abs{\Om_{n}}}}}}^2 }
\\ &=
\frac{1}{a_n}
\sum_{\om \in \Omega_n}
\ew \kl{ \mabs{ \Para_n\skl{\om} - \soft\mkl{\Data_n\skl{\om}, T\mkl{\alpha_n, \abs{\Om_{n}}}}}^2 }\,.\hspace{4em}
\end{aligned}
\end{multline*}
Now we can proceed similar to~\cite{DonJoh94} (see also~\cite{Joh11,Mal09}) to estimate the  mean square errors
$\ew \mkl{\sabs{ \Para_n\skl{\om} -  \soft\mkl{\Data_n\skl{\om},T \mkl{\alpha_n, \abs{\Om_n}}}}^2}$ of one-dimensional soft-thresholding.

To that end we use the  risk estimate of~\cite[Section~2.7]{Joh11}
for one-dimensional soft-thresholding, which states the following:
If  $y  \sim N\kl{\mu,1}$ is a normal random variable with mean  $\mu \in \R$ and unit variance, then
\begin{equation}\label{eq:risk-1d}
\kl{\forall T >0}
\qquad
\ew \kl{ \abs{\mu - \soft \kl{y, T} }^2}
\leq e^{-T^2/2} + \min \set{1+T^2, \mu^2}\,.
\end{equation}
For our purpose we  apply the risk estimate~\req{risk-1d} with
threshold $T = T \mkl{\alpha_n, \abs{\Om_n}}$.
The definition of the threshold $T \mkl{\alpha_n, \abs{\Om_n}}$  in \req{alphathresh} immediately
yields the estimate
\begin{equation*}
\frac{T \mkl{\alpha_n, \abs{\Om_n}}^2}{2}
\geq
\log\abs{\Om_n}
-
\log \log \mkl{1/\skl{1-\alpha_n}}
-
\frac{\log\log \abs{\Om_n}
+ \log \pi}{2}\,.
\end{equation*}
Inserting these estimates in \req{risk-1d}  applied with the  random variables
$y = \Data_n\skl{\om}$ having mean values
$\mu = \Para_n\skl{\om}$ and using the
assumption $T \mkl{\alpha_n, \abs{\Om_n}} \leq \sqrt{2 \log\abs{\Om_n}}$ yields
 \begin{multline*}
\ew \kl{ \mabs{ \Para_n\skl{\om} -
\soft\mkl{\Data_n\skl{\om},T \mkl{\alpha_n, \abs{\Om_n}}}}^2 }
\\ \leq
\frac{\log \kl{1/\skl{1-\alpha_n}} \sqrt{\pi \log \abs{\Om_n} }}{\abs{\Om_n}}\,
+
\kl{1+ 2\log \abs{\Om_n}}
\min\set{1,  \sabs{ \Para_n\skl{\om}}^2 }\,.
\end{multline*}
Finally, summing over all
$\om \in \Om_n$ shows~\req{oracle}.

 \subsection{Proof of Theorem~\ref{thm:ti}}
 \label{ap:thm:ti}

Let $\eta = \kl{\eta\skl{t}: t\in [0,1]}$ denote a white noise
process on $[0,1]$ and consider the periodic continuous domain wavelets $\wave_{j,b}\skl{t} = 2^{j/2} \wave\mkl{2^{j} \kl{t-b} }$. We then define the  random vectors $X_n$
as inner products
\begin{equation*}
\skl{\forall j=0, \dots, \log n - 1}
\skl{\forall \ell=0, \dots,  2^jn -1 }
\quad
X_n\skl{j,\ell}
:=
\inner{\wave_{j, 2^{j}\ell/n}}{\eta} \,.
\end{equation*}
Hence the random variables $X_n\skl{j,\ell}$ are coefficients  of the white noise process $\eta$ with respect to a discrete wavelet transform, that is oversampled by factor $n$ at every scale.
Comparing this with the  definition of the translation invariant wavelet transform we see that the translation invariant wavelet coefficients $\Wave_{n,n} \noise_n$
are a subset of the elements of  $X_n$.
Hence we have
\begin{equation} \label{eq:ti-aux1}
\kl{\forall T >0}
\qquad
\wk\set{\snorm{\Wave_{n,n}\noise_n }_{\infty} \leq T }
\geq
\wk\set{\snorm{X_n}_{\infty} \leq T } \,.
\end{equation}

We proceed by computing the correlations of $X_n\skl{j,\ell}$
for some fixed scale index. Since $\eta$ is a white noise process, the definition of $X_n$ and some elementary manipulations shows that, for all $j \in \set{0, \dots, \log n -1}$ and all indices  $\ell$, $\ell' \in \sset{ 0, \dots  , 2^{j}n - 1}$, we have
\begin{multline*}
\cov\mkl{X_n\skl{j,\ell}, X_n\skl{j,\ell'} }
=
\minner{\wave_{j,\ell}}{\wave_{j,\ell'}}
=
2^{j}\int_{\R}
\wave\skl{2^j t - \ell/n}
\wave\skl{2^j t - \ell'/n}
dt
\\=
2^{j}\int_{\R}
\wave\skl{2^j t}
\wave\skl{2^j t -  \skl{\ell'- \ell}/n} dt
=
\int_{\R} \wave\skl{-t} \wave\skl{\skl{\ell- \ell'}/n-t} dt
=
\skl{\bar\wave\ast  \wave } \skl{\skl{\ell-\ell'}/n}\,.
\end{multline*}

Next we construct  a random vector $Y_n$ with
the same index set and pointwise smaller correlations.
To that end, for every given $j$, we group the index set
$\sset{0, \dots 2^jn -1}$ into $2^j$ blocks
$B_{j,k} = \sset{kn, \dots, \kl{k+1}n -1}$ for any
$k \in \sset{0, \dots, 2^j-1}$.
We  denote
$\kappa := \bar\wave \ast \wave$ and define the matrix
\begin{equation*}
	\bar{\kappa}_{n}
	\skl{\skl{j,\ell},\skl{j',\ell'}}
	:=
	\begin{cases}
		\kappa \kl{ \skl{\ell-\ell'}/n}
		& \text{ if } j=j' \text{ and } \skl{\ell,\ell'} \in
		\bigcup_{k}B_{j,k}\times B_{j,k}\\
		0 &\text{ otherwise } \,.
	\end{cases}
\end{equation*}
 Hence we have $\bar{\kappa}_{n} \kl{\skl{j,\ell},\skl{j',\ell'}}
 =\cov\mkl{X_n\skl{j,\ell}, X_n\skl{j',\ell'} }$
 if $j = j'$ and the indices $\ell,\ell'$ are in the  same block $B_{j,k}$,
 and the correlations  of $\bar{\kappa}_{n}$ are zero otherwise.
 Moreover $\bar{\kappa}_{n}$ is obviously symmetric and positive semi-definite and hence there exists a standardized normal   random vector $Y_n$ whose covariance matrix is given  by $\bar{\kappa}_{n}$.
  By construction of $\bar\kappa_n$, the covariances
 $\sabs{\cov\skl{X_n\skl{j,\ell}, X_n\skl{j,\ell'} }}$ pointwise dominate the covariances
 $\mabs{\bar{\kappa}_{n}\skl{\skl{j,\ell},\skl{j',\ell'}}}$.
 Hence,  Lemma~\ref{lem:sidak} implies
 \begin{equation} \label{eq:ti-aux2}
\kl{\forall T >0}
\qquad
\wk\set{\snorm{X_{n}}_{\infty} \leq T }
\geq
\wk\set{\snorm{Y_n}_{\infty} \leq T } \,.
\end{equation}
Inspecting Equations~\req{ti-aux1} and~\req{ti-aux2}
shows that it remains to compute the asymptotic
distribution of $\snorm{Y_n}_{\infty}$.

To that end recall that
$\cov\kl{X_n\skl{j,\ell}, X_n\skl{j,\ell'}} =
\kappa \kl{ \skl{\ell-\ell'}/n}  $ are densely sampled
values of the autocorrelation function of the mother
wavelet. This in particular implies that any block in $Y_{n}$ has the same distribution. Moreover, due to the independence of the blocks this yields
\begin{align*}
	\wk \set{\snorm{Y_n}_\infty \leq T}
	&=
	\wk \set{\max{\sabs{Y_n \skl{0,\ell}:
	\ell = 0, \dots,  n - 1 }  \leq  T} }^n
	\\
	&=
	\kl{ 1- \wk \set{\max{\sabs{Y_n \skl{0,\ell}:
	\ell = 0, \dots,  n - 1 }}  >  T} }^n
	\\
	&=
	\kl{ 1-
	\wk \set{\max{\sabs{\sinner{ \wave_{0,\ell/n}}{\eta}}:
	\ell = 0, \dots,  n - 1 }}  >  T }^n
	\\
	&=
	\kl{ 1-
	\wk \set{\max{\abs{X\skl{\ell/n}}:
	\ell = 0, \dots,  n - 1 }}  >  T }^n
	\,.
\end{align*}
Here $X = \set{X\skl{t}: t\in [0,1]}$ is defined by
$X\skl{t} := \inner{ \wave_{0,t}}{\eta}$.
One easily verifies that  $X$ is a mean square differentiable normal  process  having  covariance function $\kappa \skl{t}$.
Moreover the vector  $Y_n \skl{0,\ell} = X\skl{\ell/n}$ consist of $n$ equidistant  values of that process  inside
the unit interval.
Hence for any sequence of thresholds $T_n$ that tends to infinity as $n \to \infty$   in a sufficiently slowly manner, one has the asymptotic relations (which follow from standard result of
continuous extreme value theory  \cite{LeaLinRoo83})
\begin{align*}
\wk \set{\max \set{\abs{X \skl{\ell/n}}:
\ell = 0, \dots,  n - 1 }  >  T_n}
&\sim
\wk \set{\max \set{\abs{X\skl{t}} : t \in [0,1]}  >  T_n}
\\
\wk \set{\max \set{\abs{X\skl{t}} : t \in [0,1]}  >  T_n}
&\sim
2\wk \set{\max \set{X\skl{t} : t \in [0,1]}  >  T_n}
\\
\wk \set{\max \set{X\skl{t} : t \in [0,1]}  >  T_n}
&
\sim
c/\skl{2\pi}\exp\kl{-T_n^2/2}
\,.
\end{align*}
Now fix any $z \in \R$ and define the sequence
$T_n  := \mkl{2\skl{\log n +z + 2\log\skl{c/\pi}}}^{1/2}$. Then the definition of $T_n$ immediately yields  $ \exp\mkl{-T_n^2/2}  =  \pi/\kl{cn}
\exp\skl{-z}$.
Consequently, by collecting the above estimates, we have
\begin{equation*}
\lim_{n \to \infty}
\wk \set{\norm{Y_n}_\infty \leq T_n}
=
\lim_{n \to \infty}
\kl{1 - \frac{c}{\pi} e^{-T_n^2/2}}^n
=
\lim_{n \to \infty}
\kl{1 - \frac{e^{-z}}{n} }^n
=
\exp\kl{-e^{-z}  } \,.
\end{equation*}

Finally, a simple Taylor series approximation of the square
root shows the asymptotic relation
\begin{equation*}
T_n
=
\sqrt{2\log n} + \frac{x+ \log \kl{c/\pi}}{\sqrt{2\log n}} + o \kl{1/\sqrt{2\log n}} \,.
\end{equation*}
Recalling, for the last time, that $o \mkl{1/a_n}$ terms can
be omitted in the rescaling of extreme value distributions finally shows
\begin{equation*}
	\wk \set{\norm{Y_n}_\infty \leq \sqrt{2\log n} + \frac{x+ \log \kl{c/\pi}}{\sqrt{2\log n}}}
	\to \exp\kl{-e^{-z}  } \,,
\end{equation*}
and concludes the proof.

\providecommand{\noopsort}[1]{}

\end{document}